\documentclass[letterpaper,11pt,reqno]{amsart}

\usepackage[utf8]{inputenc}

\makeatletter
\usepackage{amssymb}
\usepackage{latexsym}
\usepackage{amsbsy}
\usepackage{amsfonts}
\usepackage{hyperref}
\usepackage{graphicx}
\usepackage{enumerate}
\usepackage{enumitem}
\usepackage{mathtools}
\usepackage{color}

\usepackage{dsfont}

\usepackage{tikz}

\def\marginpar#1{\ignorespaces}

\def \E{\mathbb{E}}

\def \N{\mathbb{N}}

\def \P{\mathbb{P}}

\def \R{\mathbb{R}}

\newcommand{\leqp}{\preceq_+} 
\newcommand{\leqc}{\preceq_{C}}
\newcommand{\leqe}{\preceq_{E}}
 
\newcommand{\leqs}{\preceq_{\text{\rm sto}}}

\newcommand{\leqmaj}{\preceq}

\DeclareMathOperator\argmin{\arg \min}
\DeclareMathOperator\dom{dom}
\DeclareMathOperator\idom{idom}
\DeclareMathOperator\bdom{bdom}

\newtheorem{theorem}{Theorem}[section]
\newtheorem{lemma}[theorem]{Lemma}
\newtheorem{proposition}[theorem]{Proposition}
\newtheorem{corollary}[theorem]{Corollary}
\newtheorem{definition}[theorem]{Definition}
\newtheorem{remark}[theorem]{Remark}
\newtheorem{example}[theorem]{Example}
\numberwithin{equation}{section}
\makeatother
\begin{document}
\title[]{The football model, stochastic ordering and martingale transport}

\author[Gaoyue Guo]{{Gaoyue} Guo}
\address{Université Paris-Saclay CentraleSupélec, Laboratoire MICS and CNRS FR-3487. 
} \email{gaoyue.guo@centralesupelec.fr}
\author[Nicolas Juillet]{{Nicolas} Juillet}
\address{(1) Universit\'e de Haute Alsace, IRIMAS (Institut de Recherche en Informatique, Math\'ematiques, Automatique et Signal) UR 7499, F-68\,100 Mulhouse, France (2) Université de Strasbourg.}\email{nicolas.juillet@uha.fr}
\author[Wenpin Tang]{{Wenpin} Tang}
\address{Department of Industrial Engineering and Operations Research, Columbia University, New York, USA.
} \email{wt2319@columbia.edu}

\date{\today} 
\begin{abstract}
Tournaments are competitions between a number of teams, 
the outcome of which determines the relative strength or rank of each team.
In many cases, the strength of a team in the tournament is given by a score. 
Perhaps the most fundamental mathematical result in the theory of random tournaments is Moon's theorem, which provides a necessary and sufficient condition for a feasible score sequence via majorization. 
To give a probabilistic interpretation of Moon's result,
Aldous and Kolesnik introduced the football model,
the existence of which gives a short proof of Moon's theorem.
However, the proof of Aldous and Kolesnik is ``noncanonical'',
leading to the question of a canonical construction of the football model. 
The purpose of this paper is to provide explicit constructions of the football model
with an additional stochastic ordering constraint,
which can be formulated by martingale transport.
Two solutions are given: one is by solving an entropy optimization problem via Sinkhorn's algorithm,
and the other relies on the idea of shadow couplings. 
It turns out that both constructions yield the property of strong stochastic transitivity.
Nontransitive version of the football model is also considered.
\end{abstract}

\maketitle




\section{Introduction}
\label{sc1}

A tournament refers to pairwise competitions between a number of teams (or players),
all participating in a sport or game, 
in order to determine the winner or to produce a ranking of the teams.
There has been a long history of research in tournaments and their rankings,
including psychology \cite{Thur27, Thur31},
game and economic theory \cite{AH81, La97, LR81},
combinatorial and graph theory \cite{GS71, HM66, Moon68},
applied probability \cite{YF17, Aldous21, Laj91}, 
statistics \cite{Dv59, Kendall55, Mallows57},
and more recently, large language models via direct preference optimization
\cite{Chen24, DPO}.
In everyday language, a tournament often means 
a {\em single-elimination} or {\em knockout tournament}. 
This paper focuses on the $n$-team {\em round-robin tournament}, 
where each team competes against every other team. 
This often refers to {\em pairwise comparisons} in the statistics literature.
A  common metric for evaluating the strengths of the teams in the tournament 
is the {\em score sequence}, or simply, the {\em score} \cite{Landau53, Mac20}.
We will consider the football model \cite{AK22},
which is an alternative to the famous Bradley--Terry model \cite{BT52}.
The main contribution of this work is to provide 
two explicit constructions of the football model that satisfies a stochastic ordering constraint,
exploiting various probabilistic techniques such as 
stochastic ordering, martingale transport and entropy optimization.
This answers a question of Aldous and Kolesnik (see Subsection \ref{sc15} for details).

\subsection{Generalized tournament matrices and Moon's theorem}
To provide context, let us first introduce the notion of {\em tournament matrices} and {\em scores}.
Consider an $n$-team tournament where each team competes $N$ times against every other opponent. 
For $i \ne j$, let $p_{ij} \in [0,1]$ be a number representing the relative strength between the teams $i$ and $j$.
It can be interpreted as:
\begin{itemize}[itemsep = 3 pt]
\item
the proportion of the number of wins of team $i$ over team $j$ in $N$ games;
\item
the probability that team $i$ wins over its opponent $j$ in a random game.
\end{itemize}

\begin{definition}[Generalized tournament matrices and scores]
\label{def:GTM}
Let $n \ge 1$.
\begin{enumerate}[itemsep = 3 pt]
\item
Denote by $\mathcal{G}_n$ the set of matrices $P=(p_{ij})_{1\leq i\neq j\leq n}$,
where $p_{ij}\in [0,1]$ and $p_{ij}+p_{ji}=1$ for each $i \ne j$.
Such matrices are called the \emph{generalized tournament matrices}  of dimension $n$. 
(Note that the diagonals are undetermined.)
Moreover, if $p_{ij}\in \{0,1\}$ for each $i\neq j$, then $P$ is a \emph{tournament matrix}.
\item
Denote by $\mathcal{G}'_n$ the set of matrices $P=(p_{ij})_{1\leq i, j\leq n}$,
where $p_{ii}=1/2$ for each $i = 1, \ldots, n$ and $(p_{ij})_{1\le i\neq j\le n}\in\mathcal{G}_n$.
\item
For $\mathbf{x} = (x_1, \ldots, x_n) \in [0,+\infty)^n$, the subset $\mathcal{G}_n(\mathbf{x})\subset \mathcal{G}_n$ (resp. $\mathcal{G}'_n(\mathbf{x})\subset \mathcal{G}'_n$) denotes the collection of matrices $P$ such that 
\[\sum_{j\neq i}p_{ij}=x_i \quad \left(\mbox{resp. }  \sum_{j=1}^n p_{ij}=x_i\right),\]
for each $i = 1, \ldots, n$.
The vector $\mathbf{x}$ is called the \emph{(generalized tournament) score}.
\end{enumerate}
\end{definition}

In the probabilistic setting, for every $P\in\mathcal G_n(\mathbf{x})$, $\mathbf{x}$ can be interpreted as the vector of expected number of wins because $x_i=\sum_{j\neq i} p_{ij}$ for each $i$. 
By convention, for $P\in\mathcal G_n'(\mathbf{x})$, a half win is artificially added by setting $p_{ii}=1/2$ (this would correspond to a match of a team against a copy of itself). 
In the deterministic case, $N x_i$ is the number of games that team $i$ wins. 
A classical theorem of \cite{Moon63} provides a necessary and sufficient condition 
for the score $\mathbf{x}$ so that $\mathcal{G}_n(\mathbf{x})$ (or $\mathcal{G}'_n(\mathbf{x})$) is not empty.
To state Moon's result, we need the notion of {\em majorization}. 

\begin{definition}[Majorization] \label{def:maj}
Let $n \ge 1$. 
On the set of vectors $\mathbf{x} = (x_1, \ldots, x_n)$ with increasing coordinates (i.e., $x_1\leq\cdots\leq x_n$),
the partial order of majorization is defined by
\begin{align}\label{eq:def_maj}
\mathbf{x} \leqmaj \mathbf{y} \quad \text{if and only if} \quad 
\sum_{i=1}^n x_i=\sum_{i=1}^n y_i\text{ and }\sum_{i=1}^k x_i \geq \sum_{i=1}^k y_i, \text{ for each }k < n.
\end{align}
This order\footnote{On $\R^n$, it is only a preorder: the reflexivity fails. See Appendix \ref{sc6} for reminders.} extends to $\R^n$ by $\mathbf{x} \leqmaj \mathbf{y}$ if and only if $\tilde{\mathbf{x}}\leqmaj \tilde{\mathbf{y}}$,
where the vectors $\tilde{\mathbf{x}}, \tilde{\mathbf{y}}$ are the rearranged versions of $\mathbf{x}, \mathbf{y}$
(i.e., $\tilde{\mathbf{x}}=(x_{\sigma(1)},\ldots,x_{\sigma(n)})$, $\tilde{\mathbf{y}}=(y_{\sigma'(1)},\ldots,y_{\sigma'(n)})$ for some permutations $\sigma, \sigma' $)
such that the coordinates of $\tilde{\mathbf{x}}, \tilde{\mathbf{y}}$ are nondecreasing.
\end{definition}

\begin{theorem}[\cite{Moon63}]\label{thm:moon}
Let $\mathbf{x} \in \R^n$. 
\begin{enumerate}[itemsep = 3 pt]
\item
$\mathcal{G}_n(\mathbf{x})$ is nonempty if and only if $\mathbf{x} \preceq(0,1,\ldots,n-1)$.
\item
$\mathcal{G}'_n(\mathbf{x})$ is nonempty if and only if $\mathbf{x} \leqmaj (\frac{1}{2},\frac{3}{2}, \ldots,n-\frac12)$. 
\end{enumerate}
\end{theorem}

Note that a result similar to Theorem \ref{thm:moon} for tournament matrices (see Definition \ref{def:GTM} (1))
and $\mathbf{x} \in \mathbb{Z}^n$ was proved by H.\ G.\ \cite{Landau53}.
See also \cite{kolesnik2023asymptotic},  \cite{BDK24, CDF23, DK24, Sto23} for recent developments on the enumeration of score sequences. 

\subsection{Zermelo(-Bradley-Terry) model}
\label{sc12}
Definition \ref{def:GTM} is quite general. 
Several parametric models have been developed to provide further structures to $p_{ij}$'s in pairwise comparisons. 
The most popular and well-studied example is the {\em Zermelo model} \cite{Zer29},
which is widely known as the {\em Bradley-Terry model} \cite{ BT52}.

 \cite{Zer29} is arguably the first to consider the inference problem in tournaments. 
In his model, each team $i$'s strength is specified by a positive number $u_i$,
which is called the ``force" ({\em Spieltärke} in German), with $\sum_{i = 1}^n u_i = 1$.
For some $k \ge 1$,
the outcome of the $k^{th}$ game between the teams $i$ and $j$ can be represented by
a Bernoulli variable $B^k_{ij}$ with parameter $\frac{u_i}{u_i + u_j}$,
where $\{B^k_{ij}=1\}$ means that team $i$ beats team $j$ in their $k^{th}$ game,
and $\{B^k_{ij}=0\}$ indicates the other way around.
Hence, the generalized tournament matrix is specified by
$p_{ij}: = \frac{u_i}{u_i + u_j}$.
Assuming that the random variables $(B^k_{ij})_{i,j,k}$ are all independent,
the maximum likelihood estimate (MLE) is used to infer the vector parameter $\mathbf{u} = (u_1, \ldots, u_n)$.
Under an irreducibility condition (see Definition \ref{def:irred} below),
the MLE is uniquely determined by a system of $n$ equations\footnote{Though the likelihood is homogeneous of degree zero, the uniqueness is guaranteed by the constraint $\sum_{i=1}^nu_i=1$.}
\begin{equation}\label{eq:moment}
x_i = \sum_{j\neq i} \frac{u_i}{u_i+u_j},\quad \text{ for  } 1 \le i \le n.
\end{equation}
The Zermelo model was rediscovered in \cite{BT52, Ford57}.
Since then, there have been various extensions such as the Plackett-Luce model \cite{Luce, Plackett} and the Mallows model \cite{Mallows57}.
See \cite{Brad76, Stob84} for historical notes and further references. 

The Zermelo model can be reparametrized by $v_i=\log u_i$, which yields
the (generalized) linear model \cite[Section 7.5]{MN83}:
\begin{equation}
\label{eq:GLM}
p_{ij} = \frac{u_i}{u_i + u_j} = \frac{1}{1 + e^{v_j - v_i}}, \quad \mbox{for } 1 \le i\neq  j \le n.
\end{equation}
\begin{remark}\label{foot:3}
Here we provide some explanations of the reparametrization $v_i=\log u_i$. 
The idea, following Thurston and Mosteller \cite{Thur27, Thur31, Mos51}, is to associate each player $i$ with an independent random variable $X_i$, and 
the generalized tournament matrix is specified by $p_{ij} = \mathbb{P}(X_i \ge X_j)$ for $1 \le i,j \le n$.
One choice is that $X_i = v_i + V$, where $v_i$ is the indicator of strength, and $V$ is the noise. 
Hence, $p_{ij} = \mathbb{P}(W \le v_i - v_j)$,
where $W$ is distributed as $V - V'$, with $V'$ an independent copy of $V$.
In practice, $W$ is often only required to be a (symmetric) random variable. 
Specializing to the case where $W$ has the cumulative distribution function $F(t) = \frac{1}{1+e^{-t}}$
recovers the model \eqref{eq:GLM}.
Alternatively, we can take  $X_i$ to be an independent exponential 
random variable with parameter $u_i^{-1}$ to recover the model \eqref{eq:GLM}. 
\end{remark}

It is easy to check that the model \eqref{eq:GLM} enjoys the property of \emph{strong stochastic transitivity} (SST):
\begin{align}\label{eq:sst}
p_{ij}\geq1/2\text{ and }p_{jk}\geq1/2\quad \Longrightarrow\quad p_{ik}\geq \max(p_{ij},p_{jk}),
\end{align}
because the left side of \eqref{eq:sst} is equivalent to $v_i\geq v_j\geq v_k$.
This property will be the center of our study in Sections \ref{sc2} and \ref{sc5}.

\subsection{The football model} 
\label{sc13}
Recently, \cite{AK22} introduced a new parametric model for tournaments -- the football model.
The idea is in the same spirit to Remark \ref{foot:3} by associating each team with a random variable, or equivalently a probability distribution,   
for paired comparisons.
It is described as follows. 

The model is parametrized by
\begin{equation}\label{def:model}
\Theta_n :=\left\{(\mu_1,\ldots,\mu_n)\in \mathcal{P}(\{0,,\ldots,n-1\})^n: \sum_{i=1}^n \mu_i=\sum_{k=0}^{n-1}\delta_k\right\},
\end{equation}
where $\mathcal{P}(\{0,\ldots,n-1\})$ denotes the set of probability measures on $\{0, \ldots, n-1\}$, and 
$\delta_k$ is the Dirac measure on $k$.
Each team $i$ is assigned a probability measure $\mu_i$.
For every $k \ge 1$, the outcome of the $k^{th}$ game between the teams 
 $i$ and $j$ is determined as follows:
Let $X^k_i\sim \mu_i$ and $X^k_j\sim \mu_j$ be independent. 
\begin{itemize}[itemsep = 3 pt]
\item
If $X^k_i>X^k_j$, then team $i$ beats team $j$.
\item
If $X^k_i <X^k_j$, then team $j$ beats team $i$.
\item
If $X^k_i =X^k_j$, then each team is granted a probability $\frac{1}{2}$ (by external randomization) to win the game.
\end{itemize}
Here the random variables $X^k_i$ and $X^k_j$ can be interpreted as the number of goals 
scored by the team $i$ and $j$,
which explains the name ``football''. When the scores are equal, a tie-breaking method such that penalty shoot-out or coin tossing will take place and decide which team wins the game with equal probability. Specifically, let $Z^k_{ij}$ be a Bernoulli variable with parameter $\frac{1}{2}$, independent of $(X^k_i,X^k_j)$.
Set
\begin{equation*}
B_{ij}^k:=\mathds{1}_{\{X^k_i > X^k_j\}} + \mathds{1}_{\{X^k_i = X^k_j, \, Z^k_{ij} =1\}},
\end{equation*}
so $\{B^k_{ij}=1\}$ means that team $i$ beats team $j$ in their $k^{th}$ game,
and $\{B^k_{ij}=0\}$ indicates the other way around.
The corresponding tournament matrix is specified by
\begin{equation}
\label{eq:pijfootball}
p_{ij}:=\P(X_i>X_j)+\frac12\P(X_i=X_j), 
\end{equation}
where $X_i\sim \mu_i$ for $1 \le i \leq n$,
and  $(X_i)_{1 \le i\leq n}$ are pairwise independent.
The constraint $\sum_{i=1}^n \mu_i=\sum_{k=0}\delta_k$ in the definition 
\eqref{def:model} implies that the random number of goals of a  random team in a match is uniform on $\{0,1,\ldots,n-1\}$. 
This, of course, is not realistic when $n$ is large. However, it is a central element in the short probabilistic proof of Aldous and Kolesnik 
that will be recalled in Section \ref{sc14}.

\subsection{Proof of Moon's theorem via the football model} 
\label{sc14}
As pointed out in \cite{AK22}, a remarkable property of this model is that for $(\mu_1,\ldots,\mu_n)\in \Theta_n$,
the score $x_i=\sum_{j\neq i} p_{ij}$ of team $i$  after exactly one game against every other opponent is the expected number of goals (or points) $\E X_i=\int x\ d\mu_i(x)$.
To see this, let $\chi(\mu_i,\mu_j)$ be the probability for team $i$ to win over team $j$.
The map $\chi(\cdot, \cdot)$ can be extended into a linear map for signed measures with compact support. In fact, $\chi(\alpha,\beta):=\int f(y-x) d(\alpha\otimes\beta)(x,y)$,
where $f(z):=\mathds{1}_{\{z >0\}}+\frac{1}{2} \mathds{1}_{\{z =0\}}$.
Denoting by $\lambda: =\sum_{k=0}^{n-1} \delta_k$, 
we have 
$$\chi(\mu_i,\lambda)=\sum_{k=0}^n\chi(\mu_i,\delta_k)=\frac{1}{2}+\E X_i,$$ 
on one hand, and 
\begin{equation*}
\chi(\mu_i,\lambda)=\chi(\mu_i,\mu_i)+\sum_{j\neq i}\chi(\mu_i,\mu_j)=\frac{1}{2}+\sum_{j\neq i} p_{ij},
\end{equation*}
on the other hand. 

As a consequence, 
the more difficult implication in Moon's theorem (Theorem \ref{thm:moon}) can easily be proved by using a famous result commonly attributed to \cite{Str65} for the set of probability measures with a finite first moment. 
Earlier versions for discrete measures with finite support such as in Muirhead's inequality (see e.g., \cite{HLP}) are sufficient for this purpose.

\begin{theorem}[\cite{HLP, CFM, Str65}]\label{th:strassen}
Let $\rho$ and $\mu$ be two probability measures on $\R^d$ having a finite first moment. 
Then the following two conditions are equivalent:
\begin{enumerate}[itemsep = 3 pt]
\item 
The measures are in convex order $\rho \leqc \mu$, i.e., for any convex function $\varphi:\R^d\to \R$,
we have $\int \varphi d\rho \leq \int \varphi d\mu$.
\item 
There exists a pair of random variables $(X,Y)$ such that $X\sim \rho$, $Y\sim \mu$, and  $\E(Y|X=x)=x$ for $\rho$-almost every $x$.
\end{enumerate}
Moreover, for $d=1$, if $\rho$ and $\mu$ are uniform measures on $\{x_1\leq \cdots \leq x_n\}$ and $\{y_1\leq \cdots \leq y_n\}$ respectively, the conditions (1) and (2) are satisfied if and only if $\mathbf{x} \leqmaj \mathbf{y}$. 
In this case, the condition (2) is translated as follows: 
For $1 \le i \le n$, let $\mu_i$ be the conditional distribution of $Y$ given $X=x_i$.
We have  $\sum_{i=1}^n \mu_i=\sum_{j=1}^n \delta_{y_j}$ and $\int y\  d\mu_i(y)=x_i$ for each $i$.
\end{theorem}

The following proof of Theorem \ref{thm:moon} by Aldous and Kolesnik is hard to beat.
\begin{proof}[Proof of Theorem \ref{thm:moon}]
The implication $\mathcal{G}_n(\mathbf{x})$ is nonempty $\Rightarrow \mathbf{x} \leqmaj (0,\ldots,n-1)$ follows from the fact that the expected number of wins of the $k$-weakest teams must be no less than $1+\cdots+(k-1)={k \choose 2}$.

For the more difficult implication, assume that $\mathbf{x} \leqmaj (0,\ldots,n-1)$.
By Theorem \ref{th:strassen},
 there exists $(\mu_1,\ldots,\mu_n)\in \Theta_n$ such that $\int y\ d\mu_i(y)=x_i$ for each $i$. 
 In the football model, 
 $\int y\ d\mu_i(y)$ is the score of team $i$.
So the matrix $P = (p_{ij})_{1 \le i\ne j \le n}$ defined by \eqref{eq:pijfootball} is a generalized tournament matrix in $\mathcal{G}_n(x)$.
\end{proof}

\subsection{Motivation and guideline}
\label{sc15}
The above short proof relies on the observation that for $\mathbf{x} \leqmaj (0,\ldots,n-1)$,
the set
\begin{equation}
\label{eq:bary}
\Theta_n(\mathbf{x}):=\left\{(\mu_1,\ldots,\mu_n)\in \Theta_n: \int y\ d\mu_i(y)=x_i \, \mbox{ for } 1 \le i \le n\right\},
\end{equation}
is nonempty. 
This fact follows from a soft argument by applying Theorem \ref{th:strassen} that is nonconstructive.
It is natural to call for an explicit construction of $(\mu_1, \ldots, \mu_n) \in \Theta_n(\mathbf{x})$,
which was also asked in \cite{AK22}.
We quote them twice. First from the abstract:
\begin{quotation}
In particular, our proof of Moon’s theorem on mean score sequences seems more constructive than previous proofs. This provides a comparatively concrete introduction to a longstanding mystery, the lack of a canonical construction for a joint distribution in the representation theorem for convex order.
\end{quotation}
Quotation from the final ``Discussion'' section in \cite{AK22}:
\begin{quotation}
To us, the most interesting part of the bigger picture surrounding convex order is that there is apparently no “canonical” choice of joint distribution in (3), (8): proofs may be constructive but they involve rather arbitrary choices and the resulting joint distributions are not easily described. Recent literature on peacocks [7]\footnote{It is the book by Hirsch, Profeta, Roynette and Yor referred to as \cite{HP11} in the present paper.} studies continuous-parameter processes increasing in convex order, via many different constructions, and ideas from that literature might be relevant in our context.
\end{quotation}

Also note that the generic dimension of $\Theta_n(\mathbf{x})$ is $(n-1)(n-2)$, 
which is larger than the dimension $\frac{1}{2}(n-1)(n-2)$ of $\mathcal{G}_n(\mathbf{x})$,
and $n-1$ of the Zermelo model. 
The fact that the football model has more degrees of freedom is an advantage
because it allows one more flexibility in modeling, e.g., to fit nontransitive situations. 
It is also a weakness because the system of equations that identifies $\Theta_n(\mathbf{x})$
is underdetermined, which is not the case for the Zermelo model. It is worth noting that \cite{kolesnik2023coxeter} recently introduced \emph{Coxeter tournaments}, where players may collaborate and play solitaire games. They developed alternative versions of the Bradley-Terry model and the football model, with proofs that are more algebraic and geometric in nature.

The main objective of this paper is to provide explicit constructions of $(\mu_1, \ldots, \mu_n) \in \Theta_n(\mathbf{x})$
with the SST as an additional constraint. This approach provides ``canonical" representations of the football model,
thereby responding to the invitation mentioned above.
Our approach is based on martingale optimal transport,
which is a topic closely related to the peacocks ({\em Processus Croissant pour l’Ordre Convexe}). 
Two different algorithmic solutions are given:
one is obtained by solving an entropic martingale transport problem via Sinkhorn's algorithm \cite{PC19, SK67},
and the other is related to the concept of a \emph{shadow} 
that was introduced to define a class of interesting martingale transport plans \cite{BeJu16, BeJu21, BDN23}.

\medskip
{\bf Organization of the paper}: 
The remainder of the paper is organized as follows. 
In Section \ref{sc2}, we prove (in an abstract way) that 
$\Theta_n(\mathbf{x})$ contains an element that yields the SST.
Two explicit constructions of $(\mu_1, \ldots, \mu_n) \in \Theta_n(\mathbf{x})$
with the stochastic ordering constraint
are presented in Sections \ref{sc3} and \ref{sc4}.
In an opposite direction, 
we illustrate in Section \ref{sc5} that $\Theta_n(\mathbf{x})$ also has nontransitive solutions.
Finally, general theoretic clarifications on the partial orders induced by a generalized tournament matrix
are given in Appendix \ref{sc6}.

\section{Existence of an SST solution in the football model}
\label{sc2}

We have seen in Subsection \ref{sc12} that the Zermelo model \eqref{eq:GLM} enjoys the property of SST \eqref{eq:sst}.
Here we show that the SST also appears in the football model. 
In fact, for $\mathbf{x} \preceq(0,1,\ldots,n-1)$, 
$\Theta_n(\mathbf{x})$ contains an element that satisfies an additional constraint $\mu_1 \leqs \cdots \leqs \mu_n$ 
(see Proposition \ref{pro:strassen+}, with the definition of $\leqs$ being recalled at the end of this introduction). 
With this intermediate constraint,
the corresponding generalized tournament matrix enjoys the SST (see Proposition \ref{pro:increase2sst}).
The proof of Proposition \ref{pro:strassen+} relies on a result of \cite{MR01} (see also \cite[Section 2.6]{MuSt02})
concerning the existence of conditional martingale kernels that are increasing in stochastic order. 
This result is equivalent to the existence of 1-Lipschitz martingale transport plans  (see e.g., \cite{BeHuSt16} before Lemma 3.3), which is well known in the field of peacocks and martingale transport,
and has been discovered several times independently. 
For instance, \cite{Ke73} proved the existence in a nonconstructive way based on Choquet theory;
\cite{Low} is based on Hobson's approach to the Skorokhod embedding problem (SEP);  \cite{BeHuSt16} relies on Root's solution that is also not constructive;
and the sunset coupling \cite{BeJu21}  is also related to the SEP,
and to Kellerer’s solution for which it gives a more explicit construction.

Here two novel approaches are developed in the context of the football model. 
The first method is on entropy optimization in the space of martingale transport plans,
which will be detailed in Section \ref{sc3}. 
This method was proposed by \cite{Joe88},
but was applied to $P=(p_{ij})_{1\leq i\neq j\leq n}$ in $\mathcal{G}_n(\mathbf{x})$ instead of $\Theta_n(\mathbf{x})$.\footnote{\cite{AK22} also promoted Joe's approach. One main motivation of \cite{AK22} is to make a connection between Joe's result and Moon's theorem.}
The second method is a direct construction using shadows,
and will be explained in Section \ref{sc4}.

\label{page:sto} Recall that two probability measures $\mu$ and $\mu'$ on $\R$ are in stochastic order $\mu\leqs \mu'$
if and only if their cumulative distribution functions satisfy $F_\mu\geq F_{\mu'}$.
There are many equivalent criterions, e.g.,
there exists a coupling $(X,X')$ such that $X\sim \mu$, $X'\sim \mu'$ and $\P(X\leq X')=1$ (see \cite[Chapter 1]{SS07}).

\begin{lemma}\label{lem:compa}
Let $\mu$ and $\nu$ be two probability measures on $\R$ such that $\mu\leqs\nu$. 
Then for $X\sim \mu$ and $Y\sim \nu$ independent,
\begin{equation*}
\P(X>Y)+\frac12\P(X=Y)\geq 1/2.
\end{equation*}
Moreover, if $\mu\neq \nu$, then the equality is strict.
\end{lemma}
\begin{proof}
Consider a coupling $(X,X')$ such that $X\sim \mu$, $X'\sim \nu$ and $\P(X'\geq X)=1$. 
Let $Y\sim \nu$ be independent of $(X,X')$ so that $(X,Y)\sim \mu\times \nu$. 
Since $(Y,X')\sim \nu \times \nu$, we get
\begin{align*}
\frac{1}{2}=\P(Y>X')+\frac{1}{2}\P(Y=X')=\E[f(Y-X')],
\end{align*}
where 
\begin{equation}\label{eq:def_f}
f(z):=\mathds{1}_{\{z >0\}}+\frac{1}{2} \mathds{1}_{\{z =0\}}.
\end{equation}
Observing that $Y-X'\leq Y-X$ almost surely, we have:
\begin{align}\label{eq:cp}
\frac{1}{2}=\E[f(Y-X')]\leq \E[f(Y-X)]=\P(Y>X)+ \frac{1}{2}\P(Y=X).
\end{align}
Furthermore, if $\mu\neq \nu$,
the event $\{X<X'\}$ has nonzero probability. 
Note that $Y$ is independent from $(X,X')$, and has the same law as $X'$. 
Therefore, $\P(X<Y\leq X')>0$.
It follows $\P(f(Y-X')\leq 1/2, \, f(Y-X)=1)>0$,
and the inequality in \eqref{eq:cp} is strict.
\end{proof}

The following proposition is an improved version of Strassen's theorem,
following \cite{MR01}.

\begin{proposition}\label{pro:strassen+}
Let $\mathbf{x}, \mathbf{y}\in \R^n$ be such that $x_1\leq\cdots\leq x_n$ and $y_1\leq\cdots\leq y_n$ and assume $\mathbf{x}\leqmaj \mathbf{y}$. 
There exists $(\mu_1,\ldots, \mu_n)$ such that
 $\sum_{i=1}^n \mu_i=\sum_{k=1}^n \delta_{y_k}$,
 $\int x d\mu_i=x_i$ for each $1 \le i \le n$
 and $(\mu_i)_{1 \le i \le n}$ is increasing in $\leqs$.

If $\mathbf{y}=(0,1,\ldots,n-1)$, it can be formulated that there exists $(\mu_1,\ldots,\mu_n)\in \Theta_n(\mathbf{x})$ such that $(\mu_i)_{1 \le i \le n}$ is increasing in $\leqs$.
\end{proposition}
\begin{proof}
As seen in the introduction, Theorem \ref{th:strassen} has an improved version when $\rho, \mu$ are supported on $\R$.
In addition to $\E(Y|X=x)=x$ for $\rho$-almost every $x$, 
it was proved in \cite{MR01} that there exists a family $(\mu_x)_{x\in \R}$ of (regular) conditional laws
(i.e., $\E(f(Y)|\ X=x)=\int f(y)\ d\mu_x(y)$ for every positive $f$)
that is increasing in stochastic order: $x\leq x'$ implies $\mu_x\leqs \mu_{x'}$.
Exactly as in the second part of Theorem \ref{th:strassen}, 
this result translates in the discrete setting into the statement of the proposition.
\end{proof}

\begin{remark}\label{rem:equal}
In Proposition \ref{pro:strassen+},
if $\mu_i\leqs \mu_{i+1}$, 
then $\mu_i = \mu_{i+1}$ is equivalent to $x_i=x_{i+1}$. 
This can be seen by the coupling $(X_i,X_{i+1})$ such that $X_i\sim \mu_i$, $X_{i+1}\sim \mu_{i+1}$ and $X_i\leq X_{i+1}$ almost surely. 
Then $\E X_i=\E X_{i+1}$ is equivalent to $X_i=X_{i+1}$ almost surely.
\end{remark}

Combining Lemma \ref{lem:compa}, Proposition \ref{pro:strassen+} and Remark \ref{rem:equal} yields that
for each $\mathbf{x} \preceq (0,\ldots,n-1)$,
there exists $(\mu_1,\ldots,\mu_n)\in \Theta_n(\mathbf{x})$ such that
the corresponding generalized tournament matrix $P$ satisfies:
\begin{itemize}[itemsep = 3 pt]
\item
$p_{ij}\geq 1/2$ if and only if $x_i\geq x_j$.
\item
$p_{ij}=1/2$ if and only if $x_i=x_j$.
\end{itemize}
The next proposition proves the SST \eqref{eq:sst}.

\begin{proposition}\label{pro:increase2sst}
For $\mathbf{x}\leqmaj (0,\ldots,n-1)$,
let $(\mu_1,\ldots,\mu_n)\in \Theta_n(\mathbf{x})$ be specified as in Proposition \ref{pro:strassen+}.
Then $(\mu_i)_{1 \le i \le n}$ is increasing in $\leqs$,
and hence,
the generalized tournament matrix defined by \eqref{eq:pijfootball} 
satisfies the SST.
\end{proposition}
\begin{proof}
Take $(\mu_1,\ldots,\mu_n)\in \Theta_n(\mathbf{x})$ as in Proposition \ref{pro:strassen+},
and assume that $p_{kj}\geq 1/2$ and $p_{ji}\geq 1/2$.
As explained just above, we have $x_i\leq x_j\leq x_k$.
Since $\mu_i\leqs \mu_j$, 
there exist $X_i\sim \mu_i$, $X_j\sim \mu_j$ and $X_k\sim \mu_k$ 
such that $X_i\leq X_j$ almost surely and $X_k$ is independent of $(X_i,X_j)$.
With the function $f$ defined in \eqref{eq:def_f},
we have $f(X_k-X_i)\geq f(X_k-X_j)$ almost surely. Taking the expectation on both sides
 we get $p_{ki}\geq p_{kj}$. Similarly, $\mu_j\leqs \mu_k$ yields $p_{ki}\geq p_{ji}$.
Thus, $p_{ki}\geq \max(p_{ji},p_{kj})$.
\end{proof}

\begin{remark}
In view of Propositions \ref{pro:strassen+} and \ref{pro:increase2sst}, one may construct tournament matrices satisfying the SST via the (unique) strong Markov martingale coupling, denoted by $\pi$, between the discrete measures:
$$\sum_{i=1}^n\delta_{x_i} \quad\mbox{and} \quad \sum_{i=1}^n\delta_{i-1}.$$
A common approach to constructing $\pi$ relies on the fact that every discrete martingale can be embedded into a continuous-time martingale. 
More precisely, the so-called \emph{Bass martingale} $(X_t)_{0 \le t \le 1}$ satisfies ${\rm Law}(X_0, X_1) = \pi$, see, e.g., \cite{backhoff2017martingale}.
This martingale can be identified by solving a stochastic control problem, see \cite{conze2021bass, acciaio2025calibration}. However,  since 
$$\sum_{i=1}^n\delta_{x_i} \quad\mbox{and} \quad \sum_{i=1}^n\delta_{i-1}$$
are both discrete measures, convergence of the numerical schemes suggested in the aforementioned references
is not \emph{a priori} guaranteed, and approximation of this stochastic control problem becomes more technical. We propose two alternative constructions that are specifically adapted to the discrete setting.
\end{remark}

\section{The entropic construction}
\label{sc3}

In this section, we provide a construction of $(\mu_1,\ldots,\mu_n)\in \Theta_n(\mathbf{x})$
satisfying the stochastic ordering constraint $\mu_1\leqs\cdots\leqs \mu_n$
by solving an entropy optimization problem.
An iterative algorithm is given. 
For ease of presentation, 
we identify $\Theta_n$ and $\Theta_n(\mathbf{x})$,
whose elements are $(\mu_1,\ldots,\mu_n) \in \mathcal{P}(\{0,,\ldots,n-1\})^n$,
with the set of matrices 
$M=(m_{ij})_{1 \le i,j \le n}$ defined by $m_{ij}:=\mu_i(\{j-1\})$.
More precisely,
\begin{align*}
\begin{aligned}
&\Theta_n= \left\{M\in \mathcal{M}_n(\R):\ M\text{ is a doubly stochastic matrix}\right\}, \\
&\Theta_n(\mathbf{x})= \left\{M\in \Theta_n: \sum_{j = 1}^n m_{ij}(j-1)=x_i \, \mbox{ for } 1 \le i \le n\right\},
\end{aligned}
\end{align*}
where the measure $\mu_i$ is encoded as a vector of $n$ nonnegative coefficients corresponding to the $i^{\text{th}}$ row of $M$
by $\mu_i=\sum_{j=1}^n m_{ij}\delta_{j-1}$ for $1 \le i \le n$.

To proceed further, we need the following irreducibility definition,
which is implicit in the results established in \cite{Ford57, MoPu70, Zer29}.

\begin{definition}[Irreducibility condition]
\label{def:irred}
Let $P\in \mathcal{G}_n(\mathbf{x})$, and 
$\tilde{\mathbf{x}}$ be the rearranged version of $\mathbf{x}$.
The following conditions are equivalent:
\begin{enumerate}[itemsep = 3 pt]
\item 
For each (nontrivial) partition $I\cup I^c$ of $\{1,\ldots,n\}$, there exist $i\in I$ and $j\in I^c$ such that $p_{ij}>0$.
\item 
For each  $i,j\in \{1,\ldots,n\}$, 
there exist $r\leq n$ and a chain of coefficients $(i_\ell)_{\ell=1,\ldots,r}$ such that $i_1=i$, $i_r=j$, and $p_{i_\ell i_{\ell+1}}>0$ for each $\ell<r$. This condition can be interpreted as the strong connectivity in the oriented graph with vertices $\{1,\ldots,n\}$ and edges $s \to t$ open if and only if $p_{st}>0$.
\item 
There exists $r \ge 1$ such that all the entries of $P^r=\underbrace{P\cdots P}_{r\text{ times}}$ are strictly positive.
\item 
$\sum_{i = 1}^n \tilde{x}_i = \frac{n(n-1)}{2}$ and $\sum_{i = 1}^k \tilde{x}_i < \frac{k(k-1)}{2}$ for $k < n$.\footnote{Notice $\frac{k(k-1)}{2}=1+2+\cdots+k$ and look at Remark \ref{rem:minent} for a similar condition where $x\preceq y\neq(1,\ldots,n)$.}
\end{enumerate}
If one of the conditions (1)-(3) is satisfied, 
we say that $P$ is an irreducible generalized tournament matrix.
If the condition (4) is satisfied, $\mathbf{x}$ is called an irreducible score.

The following two conditions for reducible (i.e., non irreducible) generalized tournament matrices and scores are equivalent:
\begin{enumerate}[itemsep = 3 pt]
\item 
There exists a partition $I\cup I^c$ such that $p_{i j}=0$ for each $(i,j)\in I\times I^c$.
\item 
There exists $1\leq k<n$ such that $(\tilde x_1,\ldots, \tilde x_k)\leqmaj (0,\ldots, k-1)$ and $(\tilde x_{k+1},\ldots,\tilde x_n)\leqmaj(k,\ldots, n-1)$.
\end{enumerate}
In particular, if $P\in \mathcal{G}_n(\mathbf{x})$ is irreducible (resp. reducible),
then so is every $Q \in \mathcal{G}_n(\mathbf{x})$.
\end{definition}

For $\mathbf{x}\leqmaj (0,\ldots,n-1)$, define the entropy by
\begin{equation}
\label{eq:defH}
H : M \in \Theta_n(\mathbf{x})\mapsto \sum_{i,j=1}^nm_{ij}\log(m_{ij}).
\end{equation}
First, we show in Subsection \ref{sc31} that if $\mathbf{x} \leqmaj (0,\ldots, n-1)$ is an irreducible score, 
then $H$ has a unique minimizer whose entries are strictly positive. 
In Subsection \ref{sc32}, we prove that $P \in \mathcal{G}_n(\mathbf{x})$ corresponding to 
this unique minimizer satisfies the SST.
Next in Subsection \ref{sc33}, we provide an algorithm and prove its convergence to the minimizer of $H$.
Finally, we extend the results to the reducible case in Subsection \ref{sc34}.

\subsection{Minimizer of $H$}
\label{sc31}

Recall the definition of irreducible scores from Definition \ref{def:irred}. 
The following proposition is useful in proving that $H$ has a unique minimizer,
as well as the convergence of Sinkhorn's algorithm in Subsection \ref{sc33}. 
As pointed out by a referee,
this result appears to be well-known as
it corresponds to the main case in \cite[Theorem 2]{Bru84},
which was further developed by \cite{BrHwPy97,ChWo92}.
However, the following proof is different, and seems to be original.

 \begin{proposition}\label{pro:nozero}
Let $\mathbf{x}\leqmaj (0,\ldots,n-1)$ be an irreducible score (so that $\mathcal{G}_n(\mathbf{x})$ and $\Theta_n(\mathbf{x})$ are nonempty).
Then there exists $M \in \Theta_n(\mathbf{x})$ such that $m_{ij}>0$ for all $(i,j)$.
 \end{proposition}

 \begin{proof}
 We consider the partial order on $\Theta_n(\mathbf{x})$,
 for which $M=(m_{ij})_{ij}$ dominates $N=(n_{ij})_{ij}$
 if and only if $m_{ij}>0\Rightarrow n_{ij}>0$ for all $(i,j)$.
 Let $M$ be maximal for this order. 
Here we gather a few observations on $M$:
\begin{itemize}[itemsep = 3 pt]
\item 
For each row $i$, the set of indices $j$ with $m_{ij}>0$ is a nonempty (discrete) interval. 
Assume by contradiction that 
$m_{ij}=0$, $m_{ij_0}>0$ and $m_{ij_1}>0$ with $j_0<j<j_1$. 
Then there exists $i'\neq i$ such that $m_{i'j}>0$. 
For $\lambda\in (0,1)$ satisfying $\lambda y_{j_1}+(1-\lambda) y_{j_0}=y_j$,
consider $M+hM'$,
where $M'$ is a matrix with all entries zero except the six following $m'_{ij_0}=-m'_{i'j_0}=\lambda$, $m'_{ij_1}=m'_{i'j_1}=1-\lambda$ and $m'_{i'j}=-m'_{ij}=1$. 
For $h< \min(m_{ij},m_{i'j_0}/\lambda,m_{i'j_1}/(1-\lambda))$,
we have $M+hM' \in\Theta_n(\mathbf{x})$, which yields a contradiction to the maximality of $M$.
 \item 
 If the four entries $m_{ij}$, $m_{ij'}$, $m_{i'j}$ and $m_{i'j'}$ are positive,
 then for the two rows $i$ and $i'$ we have $m_{ik}>0\Leftrightarrow m_{i'k}>0$ (the positive entries are the same for the two rows). 
The proof follows the same argument as in the first point.
 \end{itemize}

For row $i$, let $J_i$ be the set of indices $j$ such that $m_{ij}>0$.
With the two observations above and the fact that $\sum_j m_{ij}y_j=x_i$,
we see that $M$ is a block matrix with the blocks going from left to right. 
Assume by contradiction that there are at least two blocks.
Then $J_1=\cdots=J_k=\{1,\ldots,\ell\}$ for some $k<n$,
and $J_{k+1}\neq J_k$.
In fact, $J_{k+1}\cap J_{k} = \emptyset$ or $\{\ell\}$.
Consider two cases $k\geq \ell$ and $k\leq \ell-1$:
For $k\geq \ell$, note that $(i\leq k\text{ and }j>\ell) \Rightarrow m_{ij}=0$. 
Thus, the sum of all the entries is larger than that of all entries $m_{ij}$ with $i\leq k$ or $j\geq \ell+1$, which is $k+(n-\ell)$.
(Here $k$ corresponds to the sum over the $k$ first rows, and $n-\ell$ is the sum over the $n-\ell$ last columns.)
Note that the sum of all the entries of $\Theta_n$ is $n$.
By analyzing the equality case, we get $k=\ell$ and $J_{k+1}\cap J_k=\emptyset$,
which contradicts the fact that $\mathbf{x}$ is irreducible because the probability that a team of the rows $0,\ldots,k$ defeats a team indexed by $k+1,\ldots,n$ would be zero (see the second part of Definition \ref{def:irred}).
For $k\leq \ell-1$, we consider the lower left of $M$:  $(i> k\text{ and }j<\ell) \Rightarrow m_{ij}=0$. 
The sum of all the entries is larger or equal to $(n-k)+(\ell-1)\geq n$, 
which corresponds to the $n-k$ last rows plus the $\ell-1$ first columns.
Again this sum is $n$, so $m_{1\ell}=\dots=m_{k\ell}=0$ which contradicts the fact that $J_1=\ldots=J_k=\{1,\ldots,\ell\}$.
\end{proof}

The following result shows that  if $\mathbf{x}$ is irreducible,
the entropy function $H$ on $\Theta_n(\mathbf{x})$ has a unique minimum point.

 \begin{proposition}\label{pro:uniqueminnonzero}
Let $\mathbf{x}\leqmaj (0,\ldots,n-1)$ be an irreducible score. 
Then the function $H$ defined by \eqref{eq:defH} has a unique minimizer,
whose coefficients $m_{ij}$'s are all strictly positive.
\end{proposition}

\begin{proof}
The function $H$ is continuous on the compact set $\Theta_n(\mathbf{x})$
(with the convention $0 \times \log 0 = 0$).
Moreover, it is strictly convex
so $H$ has a unique minimizer denoted by $M^{(0)}$.

Suppose by contradiction that one of the entries of $M^{(0)}$ is zero.
By Proposition \ref{pro:nozero}, 
let $M^{(1)} \in \Theta_n(\mathbf{x})$ with all strictly positive entries. 
For each $0 \le \lambda \le 1$, $M^{(\lambda)}:=\lambda M^{(1)}+(1-\lambda)M^{(0)} \in \Theta_n(\mathbf{x})$.
Because all the coefficients of $M^{(1)}$ are strictly positive,
the derivative of $\lambda\in [0,1]\mapsto H(M^{(\lambda)})$ at $\lambda=0^+$ is $-\infty$.
This contradicts the fact that $M^{(0)}$ minimizes $H$.
\end{proof}

\begin{remark}\label{rem:minent}
~
\begin{enumerate}[itemsep = 3 pt]
\item
Minimizing $P\in \mathcal{G}_n(\mathbf{x})\mapsto \sum_{ij} p_{ij}\log(p_{ij})$ (instead of $H:M \mapsto \sum_{ij} m_{ij}\log(m_{ij})$)
is used by \cite{Joe88} to construct a generalized tournament matrix of the Zermelo-Bradley-Terry model. 
Similar to Proposition \ref{pro:nozero}, we can show that if $\mathbf{x}$ is irreducible,
then $\mathcal{G}_n(\mathbf{x})$ contains some $P$ with all strictly positive entries. 
The argument is as follows: 
let $P \in \mathcal{G}_n(\mathbf{x})$ with the maximal number of nonzero entries.
Assume by contradiction that  $p_{i_1i_0}=0$, so $p_{i_0i_1}=1$.
By Definition \ref{def:irred}, there exists a chain $p_{i_1i_2},\ldots,p_{i_{N-1}i_N},p_{i_N,i_0} > 0$. 
Now we operate as follows to get $p_{i_1i_0}>0$ and guarantee $P$ stays in $\mathcal{G}_n(\mathbf{x})$: 
for sufficiently small $h$, replace $p_{i_ji_{j+1}}$ with $p_{i_ji_{j+1}}-h$, with the convention $i_{N+1}=i_0$. 
Hence, $p_{i_{j+1}i_j}$ is replaced with $p_{i_{j+1}i_j}+h$.
It suffices to use the infinite derivative of $p \mapsto p \log p$ at $p = 0^+$ as in Proposition \ref{pro:uniqueminnonzero}
to conclude. 
Note that this argument also works for other functions $f$, provided that the derivative of $f$ at $0^+$ is $-\infty$.
\item 
As was pointed out by a referee,
 both the entropic argument and Joe's theorem can be generalized to non round-robin tournaments. 
 This has been done in \cite[Theorem 7]{IsIyMcK00}, 
 where the complete graph case is equivalent to Joe's theorem.
\item
Proposition \ref{pro:uniqueminnonzero} can be easily generalized to irreducible pairs $\mathbf{x}\leqmaj \mathbf{y}$, 
where $\mathbf{y}=(k, k+1,\ldots, \ell-1, \ell)$ is different from $(0,\ldots,n-1)$. 
Definition \ref{def:irred} (4) can be adapted as follows:
$\sum_{i = 1}^{n'} \tilde{x}_{k-1+i} = \sum_{i = 1}^{n'} \tilde{y}_{k-1+i}$ and $\sum_{i = 1}^k \tilde{x}_i < \sum_{i = 1}^k \tilde{y}_i$ for $k < n'$.
 (with $n'=\ell-k+1$).
In this case, $\Theta_n(\mathbf{x})$ is replaced by the set of double stochastic matrices with 
$\sum_{j=1}^{n'} m_{ij}y_{k-1+j}=x_{k-1+i}$ for each $1 \le i \le n'$.
\end{enumerate}
\end{remark}

\subsection{Property of SST}
\label{sc32}

The following theorem shows that for an irreducible $\mathbf{x}$, the minimizer of $H$ on $\Theta_n(\mathbf{x})$
satisfies the stochastic ordering constraint,
and hence, the corresponding generalized tournament matrix 
enjoys the property of SST.

\begin{theorem}\label{them:entropty_sst}
Let $\mathbf{x}\leqmaj (0,\ldots,n-1)$ be an irreducible score,
and $M$ be the (unique) minimizer of $H$ on $\Theta_n(\mathbf{x})$.
Then $(\mu_i)_{1 \le i \le n}$ is increasing in stochastic order,
and the generalized tournament matrix corresponding to $M$ (defined by \eqref{eq:pijfootball}) satisfies the SST.
\end{theorem}

\begin{proof}
Since $\mathbf{x}$ is irreducible, $H$ has a unique minimizer on $\Theta_n(\mathbf{x})$ by Proposition \ref{pro:uniqueminnonzero}.
Moreover, the minimizer is in the interior (of the affine linear space spanned by $\Theta_n(\mathbf{x}))$.
By the Karush-Kuhn-Tucker theorem \cite[Section 5.5.3]{BV04},
we can neglect the constraints  $m_{ij}\geq 0$,
and only consider those on the rows with multipliers $a_i$'s, 
on the columns with multipliers $b_j$'s, and on the barycenters with multipliers $c_i$'s. 
Differentiating the Lagrangian, we get:
\begin{align*}
(1+\log(m_{ij}))+a_i+b_j+c_iy_j=0, \quad \mbox{where } y_j=j-1.
\end{align*} 

So we have $m_{ij}=e^{-a_i-1}e^{-b_j}e^{-c_iy_j}$,
and the probability vector $\mu_i=(m_{i1},\ldots,m_{in})$ is specified by the vector $(e^{-b_1},\ldots,e^{-b_n})$ modulated by entrywise product with the shape vector $(e^{-c_i y_1},\ldots,e^{-c_i y_n})$ and normalized by the constant $e^{-a_i}$. 
For ease of presentation,
we fix $(b_1,\ldots,b_n)$, 
let $\bar \mu=(e^{-b_1},\ldots,e^{-b_n})$, and define $\bar \mu(s)$ by
\begin{equation*}
\bar \mu(s)=(\bar \mu_1(s),\ldots,\bar \mu_n(s))=\left(\bar \mu_1 e^{sy_1},\ldots, \bar \mu_n e^{sy_n}\right).
\end{equation*}
(Note that $\bar \mu(0)=\bar \mu$, and recall $y_j=j-1$.) 
Let $C(s)=\sum_{j=1}^n \bar \mu_j e^{sy_j}$,
so that $\mu_i=C(-c_i)^{-1}\bar \mu(-c_i)$ for $1 \le i \le n$.
Now for $s<t$ and $1\leq k< n$, we have:
\begin{align*}
\frac{\sum_{j=1}^k \bar \mu_j(t)}{\sum_{j=1}^k \bar \mu_j(s)}\leq e^{(t-s)y_k}<e^{(t-s)y_{k+1}}\leq \frac{\sum_{j=k+1}^n \bar \mu_j(t)}{\sum_{j=k+1}^n \bar \mu_j(s)}.
 \end{align*}
 Therefore,
 \[\frac{\sum_{j=1}^k \bar \mu_j(t)}{\sum_{j=k+1}^n \bar \mu_j(t)}< \frac{\sum_{j=1}^k \bar \mu_j(s)}{\sum_{j=k+1}^n \bar \mu_j(s)}\]
 so that 
 \[g_k: s\mapsto C(s)^{-1}\sum_{j=1}^k\bar \mu_j(s)=\frac{\sum_{j=1}^k \bar \mu_j(s)}{\sum_{j=1}^k \bar \mu_j(s)+\sum_{j=k+1}^n \bar \mu_j(s)}=\left(1+\frac{\sum_{j=k+1}^n \bar \mu_j(s)}{\sum_{j=1}^k \bar \mu_j(s)}\right)^{-1}\] is strictly increasing.
Comparing the cumulative distribution functions of 
\begin{align}\label{eq:norm}
\mu(s)= \sum_{j=1}^n C(s)^{-1}\bar \mu_j(s)\delta_{y_j} \quad \mbox{and} \quad \mu(t)= \sum_{j=1}^n C(t)^{-1}\bar \mu_j(t)\delta_{y_j},
\end{align}
we see that $\mu(s)\leqs \mu(t)$. Recall that $\mu_i=C(-c_i)^{-1}\mu(-c_i)$ for some $c_i\in \R$. As a result, 
the probability measures $\mu_1,\ldots,\mu_n$ are totally ordered (in stochastic order),
and are in the same order as their  barycenters $x_i=\int y\ d\mu_i(y)$.
Hence, $(\mu_i)_{1 \le i \le n}$ is increasing in stochastic order.
The property of SST follows readily from Proposition \ref{pro:increase2sst}.
\end{proof}

\subsection{Sinkhorn's algorithm and convergence}
\label{sc33}

In this subsection, we provide a numerical scheme to approximate the minimizer of the entropy in the football model with the stochastic ordering constraint (and the SST) as in Theorem \ref{them:entropty_sst}. This scheme is closely related to the one  for the entropic  optimal transport with discrete marginals, which was initiated in the pioneering work of Benamou, Carlier, Cuturi, Nenna and Peyré \cite{benamou2015iterative}. Our situation is different due to the martingale constraint. Thus, we need to adapt the existing theory carefully, depending on whether $\mathbf{x}$ is irreducible or not (see Subsection \ref{sc34}).\footnote{We point out that some similar but different algorithms have been considered for martingale optimization problems satisfying either marginal or expectation constraints \cite{Av98, MH99}.}
Our presentation closely follows the one by \cite{BaLe00}. 

Set  $C:=C_1\cap C_2\cap C_3\in\mathbb R^{n\times n}$, where 
\begin{eqnarray}
    C_1&:=& \left\{M\in \mathbb R^{n\times n}_+: \sum_{j=1}^n m_{ij} =1,\, \mbox{ for all } i=1,\ldots, n\right\},\\
    C_2&:=& \left\{M\in \mathbb R^{n\times n}_+: \sum_{i=1}^n m_{ij} =1,\, \mbox{ for all }j=1,\ldots, n\right\},\\
    C_3&:=& \left\{M\in \mathbb R^{n\times n}_+: \sum_{j=1}^n (j-1)m_{ij} =x_i,\, \mbox{ for all } i=1,\ldots, n\right\}.
\end{eqnarray} 
By Proposition \ref{pro:nozero}, $C\cap (0,\infty)^{n\times n}\neq \emptyset$.

Let $h:\mathbb R\to (-\infty,\infty]$ be defined by
$h(x) = x \log x$ for $x \ge  0$, and $\infty$ otherwise.
The goal is to solve the convex optimization problem:
\begin{equation}
\min_{M\in C} \left\{ H(M):=\sum_{i,j=1}^n h(m_{ij})\right\}.
\end{equation}
To this end, we use Bregman's (pseudo-)distance $D :\mathbb R^{n\times n} \times \idom(H) \to [0,\infty]$
defined by
\begin{equation}
\begin{aligned}
D(L,M):&=H(L)-H(M)-\nabla H(M)\cdot (L-M) \\
&=\sum_{i,j=1}^n \Big[l_{ij}\log l_{ij}-m_{ij}\log m_{ij} -\big(1+\log m_{ij} \big)(l_{ij}-m_{ij})\Big],
\end{aligned}
\end{equation}
where $\dom (H)=\mathbb R^{n\times n}_+$ denotes the domain of $H$,
and its interior and boundary are denoted by $\idom (H)=(0,\infty)^{n\times n}$ and $\bdom (H)=\{M\in \mathbb R^{n\times n}_+: m_{ij} =0 \mbox{ for some } (i,j) \}$, respectively.

Set $M^0:=(m^0_{ij}\equiv 1)\in \idom(H)$,\footnote{The initialization $M^0:=(m^0_{ij}\equiv 1) \in \idom(H)$ is chosen so that 
the entropy optimization problem $\min_{L\in C} H(L)$ is equivalent to the problem 
$\min_{L\in C} D(L,M^0)$, which can be provably solved by Sinkhorn's algorithm.} so
$D(L,M^0)= H(L) - \sum_{i,j=1}^n l_{ij} + K$,
where $K\in\mathbb R$ depends only on $M^0$.
Because $\sum_{i,j=1}^n l_{ij}=n$ for $L \in C$, we have:
\begin{equation*}
\argmin_{L\in C} H(L)=\argmin_{L\in C} D(L,M^0)\quad \mbox{and} \quad \min_{L\in C} H(L)=\min_{L\in C} D(L,M^0) + n -K.
\end{equation*}
For $k \ge 4$, let $C_k:=C_{k \mbox{ \small mod } 3}$.
For $k \ge 1$, define $M_k$ as the Bregman projection of $M_{k-1}$ on $C_k$:
\begin{equation}
M^k:=\argmin_{L\in C_k} D(L,M^{k-1}),
\end{equation}
where $M^k$ is uniquely determined.  
The following theorem establishes the convergence of the sequence $(M^k)_{k \ge 0}$,
where $M^0$ is given previously and $M^k$ is uniquely defined. Finally, the convergence of our scheme is summarized as below.
\begin{theorem}
\label{thm:breg}
Let $\mathbf{x}\leqmaj (0,\ldots,n-1)$ be an irreducible score, so that $C\cap (0,\infty)^{n\times n}\neq \emptyset$.
Then
\begin{equation}
\lim_{k\to\infty} M^k=\argmin_{L\in C} D(L,M^0) =\argmin_{L\in C} H(L).
\end{equation}
\end{theorem}
\begin{proof}
Note that $\dom (h)=\mathbb R_+$,  $\idom (h)=(0,\infty)$ and $\bdom (h)=\{0\}$. It is easy to check:
\begin{enumerate}[itemsep = 3 pt]
    \item $h$ is a proper convex function. Moreover, $h$ is closed because $\{x\in \dom(h): h(x)\le \alpha\}$ is closed for all $\alpha\in\mathbb R$.
    \item $h$ is Legendre because 
    \begin{itemize}[itemsep = 3 pt]
        \item $h$ is differentiable on $\idom(h)$;
        \item $\lim_{t\to 0+} h'(x+t(y-x))(y-x)=-\infty$ for all $x\in \bdom(f)$ and $y\in \idom(h)$;
        \item $h$ is strictly convex on $\idom(h)$.
    \end{itemize}
    \item $h$ is co-finite because $\lim_{t\to \infty} \frac{h(tx)}{t}=\infty$ for all $x \ne 0$.
    \item $h$ is very strictly convex because $h''(x)>0$ for all $x\in \idom(h)$. 
\end{enumerate}  
Since $C_1, C_2, C_3$ are all affine subsets, it suffices to apply \cite[Theorem 4.3]{BaLe00} to conclude.
\end{proof}
See also \cite{PP24} for recent developments on the convergence rate of Bregman's iteration under further technical assumptions, which we do not pursue here.

Next we propose a computational scheme inspired by Theorem \ref{thm:breg}.
The key is to compute numerically, for each $M\in (0,\infty)^{n\times n}$, its Bregman projection on $C_k$. 
We distinguish three cases:
\begin{itemize}[itemsep = 3 pt]
\item
$k = 1$: We introduce the Lagrange multiplier $\lambda=(\lambda_1,\ldots, \lambda_n)\in\mathbb R^n$,
and set $\Phi(L,\lambda):=D(L,M) +\sum_{i=1}^n \lambda_i \left(\sum_{j=1}^n l_{ij} - 1\right)$.
Differentiating $\Phi$ with respect to $m_{ij}$ and $\lambda_i$ yields
$\partial_{l_{ij}} \Phi = \log(l_{ij})-\log(m_{ij}) +\lambda_i$ and 
$\partial_{\lambda_{i}} \Phi = \sum_{j=1}^n l_{ij} - 1$.
By setting these to zero, we get 
\begin{equation}
\argmin_{L\in C_1} D(L,M) = \left(\frac{m_{ij}}{\sum_{k=1}^n m_{ik}}\right)_{1\le i, j\le n}.
\end{equation}
\item
$k = 2$: The same reasoning as in the previous case yields:
\begin{equation}
\argmin_{L\in C_2} D(L,M) = \left(\frac{m_{ij}}{\sum_{k=1}^n m_{kj}}\right)_{1\le i, j\le n}.
\end{equation}
\item
$k = 3$: Define $\Phi(L,\lambda):=D(L,M) +\sum_{i=1}^n \lambda_i \left(\sum_{j=1}^n (j-1)l_{ij} - x_i\right)$,
and differentiate $\Phi$ with respect to $m_{ij}$ and $\lambda_i$ yields
$\partial_{l_{ij}} \Phi = \log(l_{ij})-\log(m_{ij}) +\lambda_i(j-1)$ and 
$\partial_{\lambda_{i}} \Phi = \sum_{j=1}^n (j-1)l_{ij} - x_i$.
Setting these to zero yields
\begin{equation}\label{eq:bary}
\argmin_{L\in C_3} D(L,M) = \left(m_{ij}r_{ij}^{j-1}\right)_{1\le i, j\le n},
\end{equation}
 where $r_{ij}$ is the unique positive root to the polynomial equation $\sum_{j=1}^n (j-1)m_{ij}r^{j-1}- x_i = 0$.
 Let $f(r):=\sum_{j=1}^n (j-1)m_{ij}r^{j-1}- x_i$. 
It is easy to see that $f$ is strictly increasing on $[0,\infty)$ with $f(0) \le 0$ and $\lim_{r\to \infty}f(r)=+\infty$. Thus, it is easy (and quick) to find a numerical root of $f$ on $[0, \infty)$ by Newton's method.
\end{itemize}
\quad To summarize, we normalize the rows ($k=1$), the columns ($k=2$) and the barycenters of the rows ($k=3$) sequentially. 
The two first operations correspond to the standard steps in Sinkhorn's algorithm. 
The third one is not exactly a barycenter normalization of the rows $\mu_i$,
because the total mass of every row changes from $1$ to another value in \eqref{eq:bary}. 
From a theoretical viewpoint, we can merge the steps 1 and 3 to a minimization problem in two variables. 
It can be checked that the minimum is specified by the proper value of $s$ for $\mu(s)$ in \eqref{eq:norm}.
Nevertheless, the current presentation is more tractable from a numerical viewpoint.

\subsection{The reducible case}
\label{sc34}

Without loss of generality, we assume $\mathbf{x}=\tilde{\mathbf{x}}$, i.e., $x_1\le  \cdots\le  x_n$.  
The reducible case can be treated similarly on each of the  irreducible components, i.e., for score subsequences $(x_{k_r+1},\ldots,x_{k_{r+1}})\leqmaj (k_r,k_r+1,\ldots,k_{r+1}-1)$, where all inequalities in \eqref{eq:def_maj} are strict except the equality of the two total sums. 
More precisely, put $k_0:=0$, and define recursively $k_{r+1}$ whenever $k_r<n$ by
\[k_{r+1}:=\min \left\{k\in \N : k_r<k\le n \text{ and } \sum_{i=1}^k x_i =\frac{k(k-1)}{2} \right\}.\]
Let $R\in \N$ be such that $k_R=n$. Obviously, $R\ge 2$.
We see from Definition \ref{def:irred} that for the reducible case,
there are different leagues with the probability for a team in a lower league to defeat a team of an upper league being zero. 
As a result, an element $M\in \Theta_n(\mathbf{x})$ must take the form:
\begin{equation*}
M=\begin{pmatrix}
M_1 & 0 &  \cdots & 0\\
0 & M_2 &  \cdots & 0\\
\vdots & \vdots & \ddots & \vdots  \\
0 & 0 &  \cdots & M_R
\end{pmatrix},
\end{equation*}
where $M_r$ has $k_{r}-k_{r-1}$ rows, the number of teams in the $r^{\text{th}}$ league. 
Multiplying $M$ on the left with the row vectors $(1,\ldots,1,0,\ldots,0)$ with $k_r$ entries $1$, and on the right with the column vector $(y_1,\ldots,y_n)=(0,1,\ldots,n-1)$, 
permits us to prove that the $M_r$'s are square matrices. 
Conversely, every block diagonal matrix made of $R$ doubly stochastic matrices that satisfy
\[\sum_{j=1}^{k_{r}-k_{r-1}} m_{ij} (k_{r-1}+j-1)=x_{k_{r-1}+1} \quad \mbox{for each } i\le k_r-k_{r-1},\]
is clearly an element of $\Theta_n(\mathbf{x})$.
Also note that the constraints in the definition of $\Theta_n(\mathrm{x})$ are separable.\footnote{The set $\Theta_n(\mathbf{x})$ is the Cartesian product of factors $\Theta_{k_r-k_{r-1}}(\mathbf{x}_r,\mathbf{y}_r)$ with $\mathbf{x}_k=(x_{k_{r-1}+1},\ldots,x_{k_r})$ and $\mathbf{y}_k=(k_{r-1}+1,\ldots,k_r)$, and the appropriate definition concerning the constraints.}

Now we see that in the minimization problem, the function $h:m\mapsto m\log m$ is zero on the off-diagonal blocks,
and $H(M) =\sum_{r=1}^R H_r(M_r)$, where each $H_r$ is the entropy function defined on $\mathcal{M}_{k_r-k_{r-1}}(\R)$. 
Since the constraints on each block are separable,
we have $R$ separate problems, each of which has a unique minimizer.\footnote{This is a slight extension of Proposition \ref{pro:uniqueminnonzero}, which is mentioned in Remark \ref{rem:minent}. The difference is that the vector $\mathbf{y}=(0,1,\dots,n-1)$ is replaced by $(k_{r-1}+1,\ldots,k_r)$.}
Numerically, we can just solve the $R$ problems separately, 
which requires us to adapt the step $3$ properly by replacing $j-1$ by the proper $y_j$.
However, the $R$ problems can be merged into one with the same steps as in the irreducible case.
It suffices to set the initial matrix with ones on the diagonal blocks, and zeros on the off-diagonal blocks.

\section{The shadow construction}
\label{sc4}

In this section, we provide another construction of $M \in \Theta_n(\mathbf{x})$ for $\mathbf{x}\leqmaj (0,\ldots,n-1)$
with the stochastic ordering constraint.
We 
present it in the form of a martingale coupling,
i.e., a probability measure $\pi^*$ on $\R^2$ 
with marginals the uniform measures $\{x_1,\ldots,x_n\}$ and $\{0,\ldots,n-1\}$,
and a martingale constraint corresponding to \eqref{eq:bary}.
Then $M = (\mu_1,\ldots,\mu_n)\in \Theta_n(\mathbf{x})$ is defined by 
$\mu_i(A)=n\pi^*(\{x_i\}\times A)$ for each $1 \le i \le n$ and Borel $A\subset \R$,
or $m_{ij}=\pi^*(\{(x_i\} \times \{j-1\}\})$ for each $1 \le i,j \le n$.

We start by recalling the terminology in \cite{BeJu16}.
For any closed set $E \subset \R$,
denote by $\mathfrak{M}(E)$ 
the set of measures on $E$.
Next we generalize the convex order (see Theorem \ref{th:strassen})
 to the {\em extended convex order}, denoted by $\leqe$.
For $\mu, \nu\in \mathfrak{M}(\R)$, 
we say $\mu\leqe\nu$ if and only if 
for any nonnegative convex function $f:\R\to\R$,
\begin{equation*}
\int_{\R}f(x)\mu(dx) \le \int_{\R}f(x)\nu(dx).
\end{equation*}
Furthermore, for $\mu, \nu\in \mathfrak{M}(\R)$,
 we say $\mu\leqp\nu$ if and only if for any nonnegative function $f:\R\to\R$,
\begin{equation*}
\int_{\R}f(x)\mu(dx) \le \int_{\R}f(x)\nu(dx).
\end{equation*}
Define the notion of shadow:
Let $\mu,\nu \in \mathfrak{M}(\R)$ such that $\mu\leqe\nu$.  
Then there exists a measure $S^{\nu}(\mu)$, called the {\em shadow of $\mu$ in $\nu$}, 
such that
\begin{enumerate}[itemsep = 3 pt]
\item
$S^{\nu}(\mu)\leqp\nu$.
\item
$\mu\leqc S^{\nu}(\mu)$.
\item
If $\eta$ is another measure satisfying (1) and (2), then we have $S^{\nu}(\mu)\leqc\eta$. 
\end{enumerate}

Now we use shadows
to construct a martingale coupling $\pi^*$.
Denoting $\frac1k\sum_{i=1}^k\delta_{z_i}$ by $U_{(z_1,\ldots, z_k)}$ we first fix
\begin{equation}\label{eq:uniform}
\mu:= U_{(x_1,\ldots, x_n)}, \quad  \nu:=U_{(0,\ldots, n-1)},
\end{equation}
and for each permutation $\sigma$ we denote $(x_{\sigma(k)})_{1 \le k\leq n}$ by $(x^{\sigma}_k)_{1 \le k\leq n}$.
By \cite[Lemma 4.13 and Example 4.20]{BeJu16},
we can carry out the following algorithm because we have $\frac{1}{n}\delta_{x_k^\sigma}\leq_E \nu_{k-1}$ at each step $k$.

\begin{enumerate}[itemsep = 3 pt]
\item 
Let $\eta^\sigma_0:=0$ and $\nu^\sigma_0:=\nu$ (we have $\eta^\sigma_0+\nu^\sigma_0=\nu$, and $\eta^\sigma_k+\nu^\sigma_k=\nu$ for each $1 \le k \le n$ with $\eta^\sigma_0\leq\eta^\sigma_1\leq\cdots$ and $\nu^\sigma_0\geq\nu^\sigma_1\geq\cdots$).
\item 
For $k = 1,\ldots, n$, define recursively:
\begin{equation}
\eta^\sigma_k := \eta^\sigma_{k-1}+S^{\nu^\sigma_{k-1}}\left(\frac{1}{n}\delta_{x^{\sigma}_k}\right), \quad
\nu^\sigma_k:= \nu^\sigma_{k-1}-S^{\nu^\sigma_{k-1}}\left(\frac{1}{n}\delta_{x^{\sigma}_k}\right).
\end{equation}
(In particular, $\eta^\sigma_k+\nu^\sigma_k=\nu$ for each $1 \le k \le n$.)
\item 
Define $\pi^{\sigma}\in\mathcal{P}(\R^2)$ (the set of probability measures on $\mathbb{R}^2$) by
\begin{align}\label{eq:pisigma}
\pi^{\sigma}(dx,dy) := \sum_{k=1}^n \delta_{x^{\sigma}_k}(dx)\otimes S^{\nu^\sigma_{k-1}}\left(\frac{1}{n}\delta_{x^{\sigma}_k}\right)(dy).
\end{align}
\item 
Define $\pi^*\in\mathcal{P}(\R^2)$ by
\begin{align}\label{eq:pistar}
\pi^* := \frac{1}{n!} \sum_{\sigma\in\mathfrak{S}(n)} \pi^{\sigma}.
\end{align}
\end{enumerate}

Recall that the output measure $\pi^*$ can be used to construct an element of  $\Theta_n(\mathbf{x})$ by
defining $\mu_i(A)=n\pi^*(\{x_i\}\times A)$, or equivalently, $m_{ij}=n \pi^*(\{x_i\}\times\{j-1\})$. 

Figure \ref{fig_shadow} illustrates the algorithm and the associated shadows for a small example ($n=3$).

\begin{figure}
\def\svgwidth{14cm}
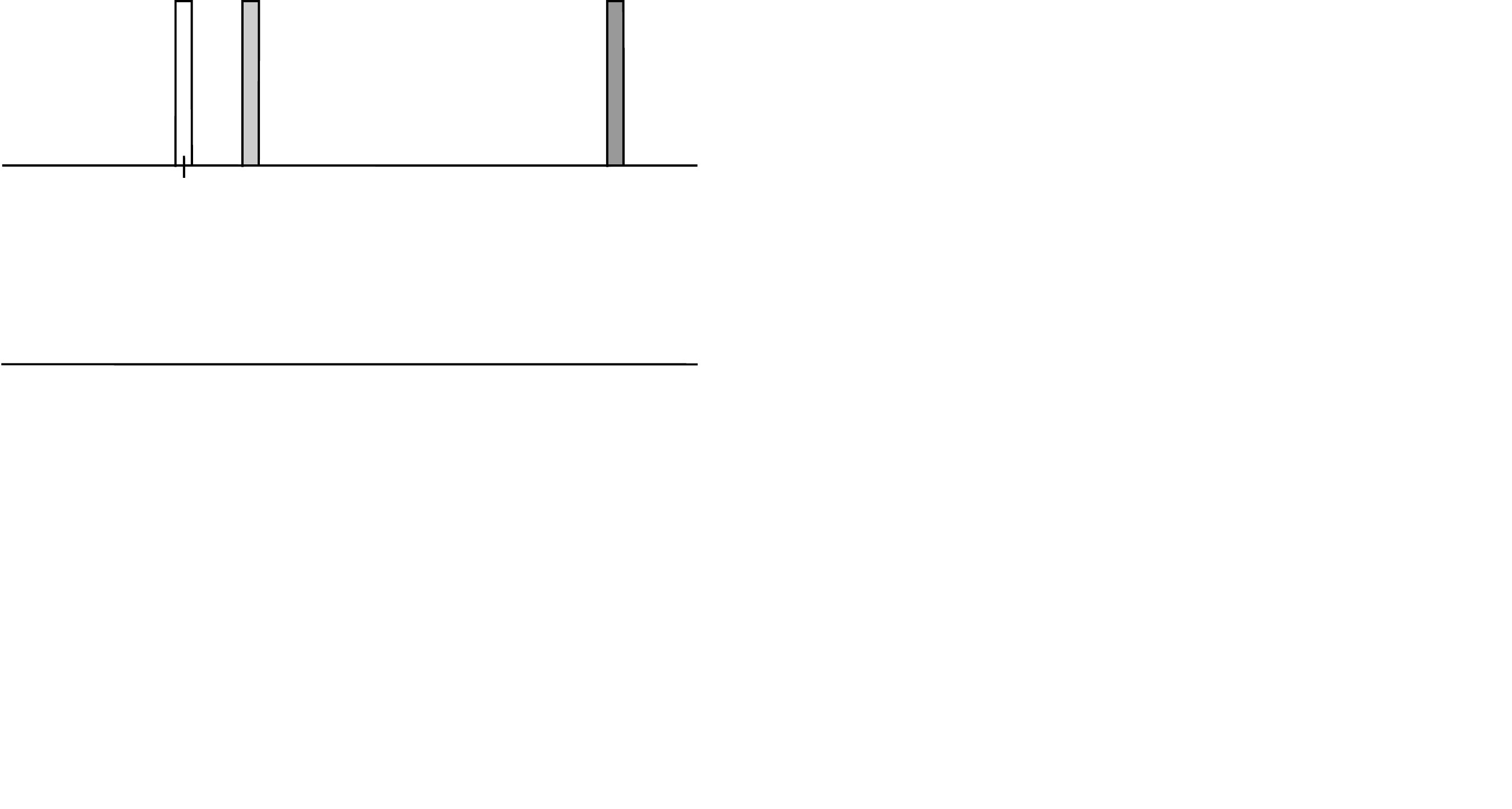
\caption{A martingale transport between $\mu=\delta_{0.5}+\delta_{0.7}+\delta_{1.8}$ and $\nu=\delta_0+\delta_1+\delta_2$ is obtained considering $3!$ permutations, two of which are represented on the two columns of this figure. On the left column the permutation $\sigma$ is the identity. $S^{\nu}(\delta_{0.5})=\frac12\delta_0+\frac12 \delta_1$ is obtained as in Remark \ref{rem:levels} for the quantile levels of $\nu$ in the interval $[0.5,1.5]$ over the whole interval $[0,3]$ of all orders. We have $\nu^\sigma_1=\nu-S^\nu(\delta_{0.5})=\frac12 \delta_0+\frac12\delta_1+\delta_2$. For the second atom $S^{\nu^\sigma_1}(\delta_{0.7})=\frac4{10}\delta_0+\frac5{10}\delta_1+\frac1{10}\delta_2$. Finally $S^{\nu^\sigma_2}(\delta_{1.8})=\nu^{\sigma}_2=\frac1{10}\delta_0+\frac9{10}\delta_2$. On the right column where $\sigma$ is different we find $S^\nu(\delta_{0.7})=\frac3{10}\delta_{0}+\frac7{10}\delta_1$ and $S^{\nu_1^\sigma}(\delta_{1.8})=\frac2{10}\delta_1+\frac8{10}\delta_2$. Finally $S^{\nu_2^\sigma}(\delta_{0.5})=\frac7{10}\delta_0+\frac1{10}\delta_1+\frac2{10}\delta_2$.}\label{fig_shadow}
\end{figure}


\begin{remark}\label{rem:probaint}
Here we give a probabilistic interpretation of the above algorithm.
If we want to know how many goals are scored by team $i$,
we pick uniformly at random a sequence $x_{\sigma(1)}, x_{\sigma(2)},\ldots$ until $\sigma(k)=i$. 
During this process, the measure $\nu$ is filled with 
\begin{equation}
\label{eq:probint}
\eta^\sigma_{k-1}=S^\nu\left(\frac{1}{n}\delta_{x_{\sigma(1)}}\right) + S^{\nu_1^{\sigma}}\left(\frac{1}{n} \delta_{x_{\sigma(2)}}\right)+\cdots+S^{\nu_{k-2}^\sigma}\left(\frac{1}{n} \delta_{x_{\sigma(k-1)}}\right)\leqp \nu,
\end{equation}
and it remains $\nu^\sigma_{k-1}=\nu-\eta^\sigma_{k-1}$.
Now we embed $\frac{1}{n}\delta_{x_{\sigma(k)}}$ in $\nu^\sigma_{k-1}$,
and obtain (up to a factor of $1/n$) a probability distribution
that underlies the random number of goals.
Finally, note that this conditional distribution is obtained by conditioning on $\sigma(1),\ldots,\sigma(k)$ rather than the entire $\sigma$.
\end{remark}

The following theorem is the counterpart to Theorem \ref{them:entropty_sst},
while the algorithm mentioned above is an alternative to the approximating sequences in Subsections \ref{sc33} and \ref{sc34}.

\begin{theorem}
For $\mathbf{x} \leqmaj(0,\ldots,n-1)$, let $\pi^*\in\mathcal{P}(\R^2)$ be defined by the above algorithm.
The marginals of $\pi^*$ are the uniform measures $\mu$ and $\nu$ defined as in \eqref{eq:uniform}.
Define 
\begin{equation*}
\mu_i:=n\pi^*(\{x_i\}\times \cdot) \, \mbox{ for } 1 \le i \le n, \quad \mbox{and} \quad 
m_{ij}:= n \pi^*(\{x_i\}\times \{j-1\})\, \mbox{ for } 1 \le i,j \le n.
\end{equation*}
Then $M = (\mu_1, \ldots, \mu_n) \in \Theta_n(\mathbf{x})$.
Moreover, $(\mu_i)_{1 \le i \le n}$ is increasing in stochastic order,
and the generalized tournament matrix corresponding to $M$ (defined by \eqref{eq:pijfootball}) satisfies the SST.
\end{theorem}

\begin{proof}
The first part of the theorem concerning the marginals and the score is in fact satisfied by $\pi^\sigma$ for each permutation $\sigma$, and hence by $\pi^*$. 
For $\pi^\sigma$,
this has been proved in \cite[Lemma 4.13 and Example 4.20]{BeJu16}.
Note that we can check directly from \eqref{eq:pisigma} that the first marginal of $\pi^*$ is 
$U_{(x_1,\ldots,x_n)}$.
We also see from this equation that $\mu_i=n  S^{\nu_{\sigma^{-1}(i)-1}^{\sigma}}\left(\frac{1}{n}\delta_{x_i}\right)$,
so $\frac{1}{n}\delta_{x_i}\leqc \frac{1}{n} \mu_i$ and $\delta_{x_i}\leqc \mu_i$, which implies $\int y d\mu_i(x)=x_i$.
(This is a basic fact in martingale transport theory, but it can also be checked by integrating $\varphi:y\mapsto \pm y$ as in Theorem \ref{th:strassen} (1).)
The fact that the second marginal of $\pi^\sigma$ is $U_{(0,\ldots,n-1)}$ is less direct,
but is still a consequence of \cite[Section 4]{BeJu16}. 
This second marginal is also $\sum^n_{k=1}S^{\nu^\sigma_{k-1}}\left(\frac{1}{n}\delta_{\sigma(k)}\right)$,
where
each of the summands is a measure of mass $\frac{1}{n}$, and is a shadow in the measure $\nu_{k-1}^\sigma\leqp\nu$,
 whose mass is $n-(k-1)$. 
The measure $\nu^\sigma_k$ is the new measure of mass $n-k$ obtained by the formula
$\nu^\sigma_k=\nu^\sigma_{k-1}-S^{\nu^\sigma_{k-1}}\left(\frac{1}{n}\delta_{\sigma(k)}\right)$ after embedding.

The second part on the fact that $(\mu_i)_{1\leq i \leq n}$ is increasing in stochastic order is purely a property of $\pi^*$ obtained by symmetrization.
Let $1 \le i < i' \le n$. We need to prove $\mu_i\leqs\mu_{i'}$, i.e.,
\begin{align}\label{eq:compare}
\frac{1}{n!}\sum_{\sigma\in \mathfrak{S}(n)}S^{\nu^{\sigma}_{\sigma^{-1}(i)-1}}\left(\frac{1}{n}\delta_{x_i}\right)\leqs\frac{1}{n!}\sum_{\sigma\in \mathfrak{S}(n)}S^{\nu^\sigma_{\sigma^{-1}(i')-1}}\left(\frac{1}{n}\delta_{x_{i'}}\right).
\end{align}
Now we explain how to derive \eqref{eq:compare} from \eqref{eq:pisigma} and \eqref{eq:pistar}.
For the sum on the left side, we embed $x_i$ (in fact the atomic measure $\frac{1}{n}\delta_{x_i})$ in $n!$ different ways depending on $\sigma$.
Here $\sigma^{-1}(i)-1$ is the number of goals $x_{\sigma(a)}$ (in fact the atomic measures $\frac{1}{n}\delta_{x_{\sigma(a)}}$)
that are embedded in $\nu$ before embedding $x_i$.
The remaining part of $\nu$, where $\frac{1}{n}\delta_{x_i}$ is embedded, is the measure $\nu^\sigma_{\sigma^{-1}(i)-1}$.
The same is done for $x_{i'}$ for the sum on the right side. 
Since stochastic ordering is preserved by addition,
it suffices to prove \eqref{eq:compare} for the sum of two terms 
$\sigma$ and $\sigma'=(i,i')\circ \sigma$, where $(i,i')$ is the transposition of $i$ and $i'$
Assume without loss of generality that $k<k'$,
with $k=\sigma^{-1}(i)$ and $k'=\sigma^{-1}(i')$. 
Also assume for simplicity $k = 1$ (we will explain how to relax this assumption later).
Let 
\begin{equation*}
\alpha=\frac{1}{n}\delta_{x_i}, \quad \beta= \frac{1}{n}\sum_{a=2}^{k'-1}  \delta_{x_{\sigma(a)}}, \quad
\gamma=\frac{1}{n}\delta_{x_{i'}}.
\end{equation*}
Note that $\alpha\leqs\gamma$.
By Lemma \ref{lem:joli} (whose proof is postponed),
we have:
 \begin{align}\label{eq:stuff}
&S^\nu\left( \frac{1}{n}\delta_{x_i}\right)+S^{\nu-S^\nu(\sum_{a=1}^{k'-1}\frac{1}{n}\delta_{x_{\sigma(a)}})}\left(\frac{1}{n}\delta_{x_{i'}}\right)
&\leqs S^\nu\left(\frac{1}{n}\delta_{x_{i'}}\right)+S^{\nu-S^\nu(\sum_{a=2}^{k'}\frac{1}{n}\delta_{x_{\sigma(a)}})}\left(\frac{1}{n}\delta_{x_{i'}}\right).
\end{align}
We recognize (without the factor $1/n!$) the contribution of $\sigma$ and $\sigma'$ to the two sums defining $\mu_i$ and $\mu_{i'}$ (one part for $k=1$ and the other for $k'$).
 By associativity of shadows from the theory in \cite{BeJu16},
 $S^\nu(\sum_{a=1}^{k'-1}\frac{1}{n}\delta_{x_{\sigma(a)}})$ and $S^\nu(\sum_{a=2}^{k'}\frac{1}{n}\delta_{x_{\sigma(a)}})$ are the measure $\eta_{k'-1}^\sigma$ and $\eta_{k'-1}^{\sigma'}$ respectively,
so we recognize in \eqref{eq:stuff} the measures $\nu^\sigma_{k}=\nu-\eta^\sigma_{k-1}$ and $\nu^\sigma_{k'}=\nu-\eta^\sigma_{k'-1}$ appearing in \eqref{eq:compare}. 
It remains the case $1<k<k'$.
Let $\xi=\sum_{a=1}^{k-1}\delta_{x_a}$. 
The two Dirac masses  $\alpha=\delta_{x_{\sigma(k)}}=\delta_{x_{\sigma'(k')}}$ and $\gamma=\delta_{x_{\sigma(k')}}=\delta_{x_{\sigma(k')}}$ are no longer embedded in
$\nu$ and $\nu-S^{\nu}(\gamma+\beta)$,
and $\nu$ and $\nu-S^{\nu}(\alpha+\beta)$ respectively,
but in $\nu':=\nu-S^{\nu}(\xi)$ and $\nu-S^{\nu}(\xi+\gamma+\beta)=\nu'-S^{\nu'}(\gamma+\beta)$,
and $\nu'$ and $\nu-S^{\nu}(\xi+\alpha+\beta)=\nu'-S^{\nu'}(\alpha+\beta)$ respectively. 
It suffices to apply Lemma \ref{lem:joli} to $\nu'$ instead of $\nu$ to conclude.
\end{proof}

We conclude this section with the following lemma concerning shadows.

\begin{lemma}\label{lem:joli}
If $\alpha+\beta+\gamma\leqe \nu$ and $\alpha\leqs \gamma$, then
\[S^{\nu}(\alpha)+S^{\nu-S^\nu(\gamma+\beta)}(\alpha)\leqs S^{\nu}(\gamma)+S^{\nu-S^\nu(\alpha+\beta)}(\gamma)\]
\end{lemma}

\begin{proof}
Since $\alpha\leqs \gamma$, it follows from \cite{Ju16} that $S^{\nu}(\alpha)\leqs S^{\nu}(\gamma)$.
 Since $\alpha+\beta\leqs \gamma+\beta$, the same argument shows 
 $S^\nu(\alpha+\beta)\leqs S^\nu(\gamma+\beta)$.
By associativity of the shadows (with $\alpha+\beta+\gamma\leqe \nu$), 
we have:
\begin{equation*}
S^{\nu-S^\nu(\gamma+\beta)}(\alpha) = S^{\nu}(\alpha+\beta+\gamma)-S^\nu(\gamma+\beta) \, \mbox{ and } \,
S^{\nu-S^\nu(\alpha+\beta)}(\gamma)=S^{\nu}(\alpha+\beta+\gamma)-S^\nu(\alpha+\beta).
\end{equation*}
Thus, $S^{\nu-S^\nu(\gamma+\beta)}(\alpha)\leqs S^{\nu-S^\nu(\alpha+\beta)}(\gamma)$,
which yields the desired result. 
\end{proof}

\begin{remark}\label{rem:levels}
Computing the distribution $\mu_i$ that gives the number of goals scored by team $i$  
seems to be difficult for large $n$, because there are $n!$ terms in the sum.
Nevertheless, it is relatively easy to simulate the random number of goals from the probabilistic interpretation of Remark \ref{rem:probaint}. 
More precisely,
the shadow $S^{\nu'}(m\delta_x)$ is the unique measure of mass $m$ and center of mass $x$ that takes the form $(F^{-1}_{\nu'})_{\#}\mbox{\em Leb}_{]\alpha,\beta]}$.
Here $F^{-1}_{\nu'}$ is any quantile function of $\nu'$, 
i.e., any nondecreasing function such that for every $\beta\in ]0,1]$,
the measure $\nu'_\beta:=(F^{-1}_{\nu'})_{\#}\mbox{\em Leb}_{[0,\beta]}$ satisfies $\nu'_\beta\leqp\nu'$ and is of mass $\beta$. 
One such example is the inverse of the cumulative distribution function $F^{-1}_{\nu'}(\gamma)=\inf\{x\in \R: \nu'(]-\infty,x])\geq \gamma\}$.
\end{remark}

\begin{remark}
Answering a question from a referee, there is at least one other algorithmic construction for the part of Theorem \ref{th:strassen} 
about $\mathbf{x}\preceq \mathbf{y}$ can be interpreted by martingale transport plans. Indeed, the inverse operations of the elementary $T$-transforms (also called Dalton transfers or Robin Hood transfers) defined  in \cite{MOA} can be interpreted by elementary martingale transport plans similar to those in \cite[Equation (6)]{Ju_seminaire}. However, these new old solutions seem to be less canonical than the two solutions detailed above.
 
 Let us give a more detailed description. A $T$-transform is a matrix that maps $\mathbf{y}$ (the more spread vector) to $\mathbf{x}$ just by modifying to entries $y_i$ and $y_j$ where $y_i<y_j$. Some positive quantity $\Lambda<y_j-y_i$ is added to $y_i$ and the same quantity is substracted to $y_j$. If we denote by $\Lambda=\lambda (y_j-y_i)$, which is always possible for some $\lambda>0$ we obtain 
\begin{align*}
x_i&=y_i+\Lambda=y_i+\lambda (y_j-y_i)\\
x_j&=y_j+\Lambda=y_j-\lambda (y_j-y_i).
\end{align*}
This can be written $\mathbf{y}=D\mathbf{x}$ for $D=\lambda I+(1-\lambda) Q$ where $Q$ is the permutation matrix encoding the transposition of the $i$ and $j$ coordinates. Of course we have $y_i<\min(x_i,x_j)\leq \max(x_i,x_j)<y_j$ and also $\frac{y_i+y_j}2=\frac{x_i+x_j}2$. The theory of majorization tells us that it is possible to shrink step by step any $\mathbf{y}$ to any $\mathbf{x}\preceq \mathbf{y}$ just composing $T$-transforms (note that we can restrict ourselves to $T$-transforms with $\lambda\leq 1/2$ so that we don't need to reorder the coordinates of $x$).

A parallel theory that may have been already written several times is to spread in place of shrinking and use the inverse transforms: for $\delta>0$ and $P$ the permutation matrix corresponding to the transposition of the $i$ and $j$ coordinates (with $x_j>x_i$), we define $D'$ by $D'=(1-\lambda)I-\delta P$. We get
\[\mathbf{x}\preceq \mathbf{x}D'=:\mathbf{y}\]
This operation can be repeated using finitely many matrices of type $D'$ to move any $\mathbf{x}$ to any $\mathbf{y}$ satisfying $\mathbf{x}\preceq \mathbf{y}$. Again $y_i<\min(x_i,x_j)\leq \max(x_i,x_j)<y_j$ and also $\frac{y_i+y_j}2=\frac{x_i+x_j}2$.

If $x_1<\cdots<x_n$ and $y_1<\cdots<y_n$ an obvious transport plan is $f(x_k)=y_k$. For $k\notin \{i,j\}$ this map satisfies the martingale transport condition because $y_i=x_i$. But $y_i<x_i$ and $y_j>x_j$. Hence, we have to modify the transport plan $\pi_f=(\mathrm{id}, f)_\# (\sum_{k=1}^n \delta_{x_k})$ to make it a martingale transport plan. This is simple: we map $\delta_{x_i}+\delta_{x_j}$ to $\delta_{y_i}+\delta_{y_j}$ in a martingale manner. For this purpose we simply replace $\pi_0=(\mathrm{id}, T)_\#(\delta_{x_i}+\delta_{x_j})=(\delta_{(x_i,y_i)}+\delta_{(x_j,y_j)})$ by 
\[\pi_1=\frac{t+s}{t+2s}\delta_{(x_i,x_i-s)}+\frac{s}{t+2s}\delta_{(x_i,x_j+s)}+\frac{s}{t+2s}\delta_{(x_j,x_i-s)}+\frac{t+s}{t+2s}\delta_{(x_j,x_j+s)}\]
where $s=x_i-y_i=y_j-x_j$ and $t=x_j-x_i$. The new transport plan $\pi:=(\pi_f-\pi_0)+\pi_1$ is a martingale transport plan that maps $\sum_{k=1}^n \delta_{x_k}$ to $\sum_{k=1}^n \delta_{y_k}$. Its construction is directly inspired by the (inverse of) the $T$-transforms. Therefore it is possible to compose this martingale transport to make it a martingale (process) that (martingale) enables transporting $\mathbf{x}$ to $\mathbf{y}$ as soon as $\mathbf{x}\preceq \mathbf{y}$.
\end{remark}

\section{Nontransitivity in the football model}
\label{sc5}

In this section, we consider (possibly) nontransitive versions of the football model. We seek for a characterization of all triples $(p_{12},p_{23},p_{31})$ attached to some $(\mu_1,\ldots,\mu_n)\in \Theta_n$ as in \eqref{eq:pijfootball},
and obtain it for every $n\geq 6$. Namely our results, Proposition \ref{pro:nonfootball} and Theorem \ref{theorem:six}, explain that the described set in $\R^3$ is exactly the same as for three independent random variables that (pairwies) coincide with zero probability.
The idea comes from the work of \cite{ST59, Tr60},
in which they characterize the triples $(\xi,\eta,\zeta)\in [0,1]^2$
that correspond to the probabilities $\xi=\P(X<Y)$, $\eta=\P(Y<Z)$ and $\zeta=\P(Z<X)$ for three independent
random variables $X,Y,Z$.
However, their result was only established under the assumption that $\P(X =Y) = \P(Y=Z) =\P(Z=X) = 0$,
which does not necessarily hold for the football model.
Here we adapt Steinhaus and Trybu\l{}a's result to the football model, where ties are allowed. 

\begin{theorem}[\cite{ST59, Tr60}]\label{them:Trybula}
Let
\begin{equation*}
A:=\left\{\left(\P(X<Y), \P(Y<Z), \P(Z<X)\right)\in [0,1]^3: \begin{aligned}
&(X,Y,Z)\text{ are independent and}\\
&\P(X=Y)=\P(Y=Z)=\P(Z=X)=0
\end{aligned}
\right\},
\end{equation*}
and $\alpha: [0,1]^2 \to \mathbb{R}_+$ be defined by
\begin{align}
\alpha(\xi,\eta)=\begin{cases}
\max\left(\frac{1-\xi}{\eta},\frac{1-\eta}{\xi},1-\xi\eta\right)&\text{for }\xi+\eta>1,\\
1&\text{for }\xi+\eta\leq 1.
\end{cases}
\end{align}
Then $A=D$, where 
\begin{equation}
D=\left\{(\xi,\eta,\zeta)\in [0,1]^3: 1-\alpha(1-\xi,1-\eta)\leq \zeta\leq \alpha(\xi,\eta)\right\}.
\end{equation}
\end{theorem}
More symmetric descriptions of $D$ are given in \cite{Su02,VuHi21}. Steinhaus and Trybu\l{}a observed a nontransitive phenomenon (which they called a ``paradox"): 
$(\xi,\eta,\zeta)\in D$ can have all its three coordinates larger that $1/2$.
For instance, the point $\left(\frac{\sqrt{5}-1}2,\frac{\sqrt{5}-1}2,\frac{\sqrt{5}-1}2\right)$ lies on the boundary of $D$,
with $\frac{\sqrt{5}-1}2\approx 0,618 > \frac{1}{2}$.

To apply Theorem \ref{them:Trybula} to the football model, 
we introduce a set $A'$ that is a priori larger than $A$:
Similar to \eqref{eq:pijfootball} it is the set of triples $(\xi',\eta',\zeta')$ with
\begin{equation}
\label{eq:XYZ}
\xi'=\P(X<Y)+\frac12\P(X=Y), \, \eta'=\P(Y<Z)+\frac12\P(Y=Z), \, \zeta'=\P(Z<X)+\frac12\P(Z=X),
\end{equation}
and $(X,Y,Z)$ independent.
Here we do not assume $\P(X=Y)=\P(Y=Z)=\P(Z=X)=0$, 
but with this additional assumption, 
we have $(\xi',\eta',\zeta')=(\P(X<Y),\P(Y<Z),\P(Y<Z))\in A$.
So $A\subset A'$.

\begin{proposition}
\label{pro:nonfootball}
With the notation above, we have $A'=D$.
\end{proposition}

\begin{proof}
By Theorem \ref{them:Trybula}, we have $D = A$. Since $A\subset A'$.
It remains to show that $A'\subset A=D$.
Let $(\xi',\eta',\zeta') \in A'$,
and $(X,Y,Z)$ be a triple of random variables satisfying \eqref{eq:XYZ}.
Let $U_X$, $U_Y$ and $U_Z$ be three independent random variables uniform on $[-\frac{1}{2}, \frac{1}{2}]$,
and also independent of $(X,Y,Z)$.
Define
\begin{equation*}
X_n=X+\frac{1}{n}U_X,\quad Y_n=Y+\frac{1}{n}U_Y,\quad Z_n=Z+\frac{1}{n}U_Z.
\end{equation*}

First we claim that $\P(X_n=Y_n)=\P(Y_n=Z_n)=\P(Z_n=X_n)=0$. 
To see this, we observe $\{X_n=Y_n\} = \left\{U_Y-U_X=n(X-Y)\right\}$,
where both sides of the equality are independent. 
Since $U_Y-U_X$ is diffuse, we get $\P(U_Y-U_X=z)=0$ for each $z\in \R$.
Integrating in $z$ with respect to the distribution of $n(X-Y)$ yields $\P(X_n=Y_n)=0$.
The other two equalities follow similarly.
Thus, 
$(\P(X_n<Y_n),\P(Y_n<Z_n),\P(Z_n<X_n))\in A=D$ by Theorem \ref{them:Trybula}.

Next we prove the following limit:
\begin{equation}\label{eq:lim}
\P(X_n<Y_n)\longrightarrow_{n\to\infty} \frac{\P(X<Y)+\P(X\leq Y)}{2}=\P(X<Y)+\frac{1}{2}\P(X=Y).
\end{equation}
Note that 
\begin{align*}
\P(X_n<Y_n)=&\P(X=Y) \, \P(X_n<Y_n\, |\, X=Y)\\
&+\P(X_n<Y_n\text{ and }X<Y)+\P(X_n<Y_n\text{ and }X>Y).
\end{align*}
Clearly, $\P(X_n<Y_n\, |\, X=Y)=\P(U_X<U_Y)=1/2$.
Since $\P(\frac{1}{n}|U_X-U_Y|\leq \frac{1}{n})=1$, we also have:
\begin{equation*}
\P\left(X-Y<-\frac{1}{n}\right) \leq \P(X_n<Y_n\text{ and }X<Y)\leq \P(X-Y<0).
\end{equation*}
By the dominated convergence theorem and the squeeze theorem,
we have $\lim_{n\to\infty}\P(X_n<Y_n\text{ and }X<Y)=\P(X<Y)$.
Similarly, $\lim_{n\to\infty} \P(X_n<Y_n\text{ and }X>Y) =0$. 
Similar to \eqref{eq:lim}, $\eta'$ and $\zeta'$ can be obtained as limits in terms of $(X_n,Y_n,Z_n)$.
Finally, since $D$ is closed, we get $(\xi',\eta',\zeta')\in A$ by sending $n \to \infty$.
\end{proof}

Now we see that the football model permits nontransitivity for some parameters in the limits of Proposition \ref{pro:nonfootball}\footnote{Note that a similar result appeared in \cite{Su02}, see Theorem 3 there, which we noticed after the paper were submitted.}.

\begin{theorem}\label{theorem:six}
For every $n\geq 6$ and each $(\xi,\eta,\zeta)\in D$,
there exists $(\mu_1,\ldots,\mu_n) \in \Theta_n$ such that the corresponding generalized tournament matrix
(defined by \eqref{eq:pijfootball}) satisfies $(p_{12},p_{23},p_{31})=(\xi,\eta,\zeta)$.
\end{theorem}

\begin{proof}
Let $(\xi,\eta,\zeta)\in [0,1]^3$ be a element of $D$ (we call it ``cyclic'' as in \cite{VuHi21}). We look for $S\subset \R$ of cardinal $6$, and three $S$-valued random variables $(X,Y,Z)$ with disjoint values such that $(\P(X<Y),\P(Y<Z),\P(Z<X)) =(\xi,\eta,\zeta)$. 
Permutations of the random variables or replacement by their opposite permits us to see that $D$ is invariant by permutation of the coordinates, and by replacement with the complementary to 1. 
Therefore, we can assume $\xi\leq \eta\leq \zeta$, and also $\xi \leq 1-\zeta$. 
By the alternative characterization in \cite[Lemma 2.4]{VuHi21},  $(\xi,\eta,\zeta)\in D$ implies that $(\xi/(1-\zeta), \eta, 0)$ is also cyclic. 
Thus, a symmetric construction to the one just after (21) in \cite{Tr60} permits to introduce three disjoint laws $\mu_{X_0},\mu_Y,\mu_Z$ on a set of five points with $(\P(X_0<Y),\P(Y<Z),\P(Z<X_0)) =(\xi/(1-\zeta), \eta, 0)$. 
Replacing $\mu_{X_0}$ by a Dirac measure $\mu_{X_1}=\delta_x$ for $x\in \R$ larger than the $5$ previous points,
one creates the cyclic vector $(0,\eta,1)$ attached to $(X_1,Y,Z)$. 
Finally, $(\xi,\eta,\zeta)$ is obtained as an element of $D$ for the triple $(X,Y,Z)$,
where $X\sim \zeta \mu_{X_1}+(1-\zeta)\mu_{X_0}$ (comparing with \cite[p.\ 328]{Tr60}, where the interpolation is not linear). Note that $Z$ is supported on one point, $X$ on two points and $Y$ on three points. An increasing homeomorphism $\varphi$ now maps $S$ on $\{-5,-4,\ldots,0\}$. Let $\mu_1,\mu_2,\mu_3$ be the distributions of $-\varphi(X),-\varphi(Y),-\varphi(Z)$ respectively. Since the supports of these measures are disjoint, 
we can complete $(\mu_1, \mu_2, \mu_3)$ with $(\mu_4,\ldots,\mu_n)$ to obtain an admissible $(\mu_1,\ldots,\mu_n)\in \Theta_n$.
\end{proof}

\appendix
\section{Partial orders via the generalized tournament matrix}
\label{sc6}

In this part, we provide further thoughts on partial orders induced by the generalized tournament matrix,
scrutinizing various results in \cite{AK22, Joe88}.
As discussed in the introduction, 
there are two ways to compare the teams $\mathcal{T}_n:=(T_i)_{1 \le i \le n}$ based on different levels of information:
\begin{itemize}[itemsep = 3 pt]
\item 
We compare $T_i$ and $T_j$ by looking at $p_{ij}$ from the generalized tournament matrix $P\in \mathcal{G}_n$ or $P\in \mathcal{G}'_n$.
\item
We compare $T_i$ and $T_j$ by looking at $x_i$ and $x_j$ from the score $\mathbf{x} = (x_1, \ldots, x_n)$,
where $x_i:=\sum_{j\neq i}^n p_{ij}$ or $x_i:=\sum_{j=1}^n p_{ij}$.
\end{itemize}
Given $P$ and $\mathbf{x}$,
there are two binary relations naturally associated with $\mathcal{T}_n$:
\begin{enumerate}[itemsep = 3 pt]
\item 
{\em Score-based relation}: $(T_i\ \leq_x\ T_j)$ if and only if $x_i\leq x_j$.
\item
{\em Results-based relation}: $(T_i\ \leq_P\ T_j)$ if and only if $p_{ij}\leq 1/2$ (if $p_{ij}$ exists), or equivalently, $p_{ij}\le p_{ji}$.
\end{enumerate}
We can also define the strict relations:
 $(T_i\ <_x\ T_j)$ if $x_i<x_j$,
 and $(T_i\ <_P\ T_j)$ if $p_{ij}<1/2$. 
Finally, $(T_i\ =_x\ T_j)$ means $(T_i\ \leq_x\ T_j)$ and $(T_i\ \geq_x\ T_j)$, i.e., $x_i=x_j$;
and $(T_i\ =_P\ T_j)$ means $(T_i\ \leq_P\ T_j)$ and $(T_i\ \geq_P\ T_j)$, i.e., $p_{ij}=1/2$.

Now we recall the definition of a partial order $\preceq$ on a set $\mathcal{T}$:
\begin{enumerate}
\item[(a)]
Reflexivity: $T\ \preceq\ T$.
\item[(b)]
Transitivity: $(T\ \preceq\ T')$ and $(T'\ \preceq\ T'')$ implies $(T\ \preceq\ T'')$.
\item[(c)]
Antisymmetry: $(T\ \preceq\ T')$ and $(T'\ \preceq\ T)$ implies $(T=T')$.
\end{enumerate}
Moreover, if all the elements are comparable, i.e., $(T\ \preceq\ T')$ or $(T'\ \preceq\ T)$ for all $T,T'\in \mathcal{T}$,
the partial order is called a {\em total order}.
If all the properties except (c) are satisfied,
the relation is called a \emph{preorder} or a \emph{total preorder}, respectively. 
It is easy to see that $\leq_x$ is a total preorder. 
In the remaining of this section, we focus on the relation $\leq_P$, and the question whether it is a total preorder. 

For a general $P\in \mathcal{G}_n$ or $P\in \mathcal{G}'_n$, 
the relation $\leq_P$ may be far from being a total preorder. 
The definition of $\mathcal{G}_n'$ (by setting $p_{ii} = \frac{1}{2}$)
ensures the reflexivity and totality (which are not satisfied by $\mathcal{G}_n$).
Here our goal is to characterize $\leq_P$ as a total preorder on $\mathcal{G}'_n$,
so it remains to consider whether $\leq_P$ is transitive.
We will review and elaborate some results evoked in the literature,
especially in \cite{AK22, Joe88}, 
and present a new observation on SST in Proposition \ref{pro:equiv_strong}.
It was explained in \cite[p.917]{Joe88} that even when $\leq_P$ is transitive,
this preorder can be different from $\leq_x$.
Nevertheless,
the notion of SST defined in \eqref{eq:sst} and studied in Proposition \ref{pro:increase2sst} and Theorem \ref{them:entropty_sst} allows us to establish
the equivalence between $\leq_x$ and $\leq_P$.
We can reformulate it as follows: for all $1 \le i,j,k \le n$,
\begin{equation*}
(T_i\ \leq_P\ T_j)\text{ and }(T_j\ \leq_P\ T_k) \quad \Longrightarrow \quad (T_i\ \leq_P\ T_k)\text{ and } p_{ik}\leq \min(p_{ij},p_{jk}).
\end{equation*}
Note that the SST is not a property of the binary relation $\leq_P$ alone,
but rather of the generalized tournament matrix $P$.
Moreover, the property of SST already contains the transitivity of $\leq_P$ in its formulation:
if $P\in\mathcal{G}_n'$, the SST implies that $(\mathcal T,\leq_P)$ is a total preorder. 
On the other hand,
the SST can a priori be satisfied by $P\in \mathcal{G}_n$ even though $T_i\leq_P T_i$ does not hold.
It is easy to see that the transitivity of $\leq_P$ for $P\in \mathcal{G}_n$ implies that for $P\in \mathcal{G}'_n$.
The contrary is false because if $p_{ij}=1/2$ for some $i\neq j$,
both relations $(T_i\leq_P T_j)$ and $(T_j\leq_P T_i)$ should enforce $(T_i\leq_P T_i)$ and $(T_j\leq_P T_j)$,
but these relations do not apply for $\mathcal{G}_n$, whose matrices have undetermined diagonal. 
It seems to us that most authors only consider $\leq_P$ for $P \in \mathcal{G}_n$.

In \cite[Theorem 2.1]{Joe88}, Joe stated that $\leq_P$ is equivalent to $\leq_x$ under the SST.
However, the one-line proof is not entirely convincing. 
We restate this result in Corollary \ref{coro:new_joe}. That is a consequence of Proposition \ref{pro:equiv_strong} where the SST is compared with the almost equivalent notion of monotonicity.

\begin{definition}[Monotonic matrix $P$]
Let $P = (p_{ij})_{1\le i\neq j\le n} \in \mathcal{G}_n$, or $P = (p_{ij})_{1\leq i,j\leq n} \in \mathcal{G}'_n$. 
The matrix $P$ is called  \emph{monotonic} if $P$ is decreasing along rows, i.e., if for every $1 \le i,j,k \le n$, inequality
$j < k$ implies $p_{ij}\geq p_{ik}$, provided $p_{ij}$ and $p_{ik}$ are defined.

Equivalently, the matrix $P$ is monotonic if $P$ is increasing down columns, i.e., for every $1 \le i,j,k \le n$, inequality
$j<k$ implies  $p_{ji}\leq p_{ki}$, provided $p_{ji}$ and $p_{ki}$ are defined.
\end{definition}

In \cite[p.849]{AK22}, Aldous and Kolesnik recalled that if the SST is satisfied,
the matrix $P$ is monotonic.
They also showed that these two properties are not equivalent, 
though the same terminology is often used interchangeably for the two close concepts.
Their counterexample is based on a matrix $P$ with some entries having value one half.\footnote{Their example is $p_{21}=p_{23}=1/2$ and $p_{13}<1/2$.}
 Example \ref{ex:nonstrong} below provides another counterexample,
where the monotonicity does not imply the SST
for a matrix $P \in \mathcal{G}_n$ with all the entries different from $\frac{1}{2}$. 

The following proposition shows the equivalence between the SST and the monotonicity.

\begin{proposition}\label{pro:equiv_strong}
Let $P=(p_{ij})_{1 \le i\neq  j \le n} \in \mathcal{G}_n$. The following statements are equivalent:
\begin{itemize}[itemsep = 3 pt]
\item[(i)] 
The relation $\leq_P$ satisfies the property of SST \eqref{eq:sst} except that we don't require $p_{ii}=1/2$ and $p_{jj}=1/2$ when $p_{ij}=p_{ji}=1/2$ for some $i\neq j$.
\item[(ii)] 
The relation $\leq_P$ satisfies the property of SST \eqref{eq:sst} for $P$ seen as a matrix in $\mathcal{G}'_n$.
\item[(iii)] 
Both conditions are satisfied:
\begin{itemize}[itemsep = 3 pt]
\item 
There exists a relabelling of $P$ by a permutation matrix $M\in \mathcal{M}_n(\R)$ such that $MPM^T$ is monotonic in $\mathcal{G}'_n$.
\item For every $i,j$, $p_{ij}=\frac{1}{2}$ implies $p_{ik}=p_{jk}$ for every $1 \le k \le n$.
\end{itemize}
\end{itemize}
Moreover, a relabelling in (iii) is admissible if and only if $i<j\Rightarrow (T_i\ \leq_P\ T_j)$.
\end{proposition}

\begin{proof}
$(i)\Leftrightarrow (ii)$. 
It is straightforward that $(ii)$ implies $(i)$.
Conversely, if $(i)$ is satisfied, extending $P$ to $\mathcal{G}'_n$ permits us to recover the full condition \eqref{eq:sst}.

$(ii)\Rightarrow (iii)$. Assume that $P\in \mathcal{G}'_n$ satisfies the SST.
We relabel the indices in such a way that $T_1\ \leq_P\ T_2\ \leq_P\cdots\leq_P\ T_n$.
(The reason why this is possible is straightforward, see \cite{AK22} for a graph theoretic proof.)
Let $i,j,k$ be three different indices with $j<k$. 
We want to prove $p_{ij}\geq p_{ik}$.
Recall that $p_{jk}\leq1/2$.
If $i<j<k$, then we have $p_{ij}\leq 1/2$, and the SST implies $p_{ik}\leq p_{ij}$. 
The case $j<k<i$ is similar.
Finally, if $j<i<k$, we have $p_{ik}\leq 1/2\leq p_{ij}$.
Now we prove the second part of $(ii)$. 
Assume $p_{ij}=1/2$, and in particular $p_{ji}\leq 1/2$.
If $p_{ik}\leq 1/2$, 
then we have $p_{jk}\leq p_{ik}\leq 1/2$. 
Since $p_{ij}\leq 1/2$,
we obtain $p_{ik}\leq p_{jk}\leq 1/2$. 
Thus, $p_{ik}=p_{jk}$. 
The case $p_{ki}> 1/2$ can be treated similarly.

$(iii)\Rightarrow (ii)$. 
Assume that there exists a relabelling such that $(p_{ij})_{i,j}\in \mathcal{G}'_n$ is monotonic and $p_{ij} = \frac{1}{2}$ implies $p_{ik}=p_{jk}$. The first property implies $p_{ij}\leq p_{ii}=\frac{1}{2}$ for every pair $(i,j)$ such that $i<j$. 
Conversely,  if $p_{ij}<1/2=p_{ii}$, we must have $i<j$. 
Assume now $T_i\ \leq_P\ T_j\ \leq_P\ T_k$.
If these relations are from $p_{ij}<1/2$ and $p_{jk}<1/2$,
we have $i<j<k$, and $p_{ik}\leq \min(p_{ij},p_{jk})$ follows from the monotonicity of $P \in \mathcal{G}'_n$.
Otherwise, $\max(p_{ij},p_{jk})=1/2$, and the second property permits us to prove $p_{ik}= \min(p_{ij},p_{jk})$.

To conclude, we make it clear when a given relabelling is admissible. 
We have already proved in part $(ii)\Rightarrow (iii)$ that if the labelling satisfies $T_1\leq_P\cdots\leq_P T_n$,
the matrix is admissible for $(iii)$. 
Conversely, assume that after relabelling the matrix $P$ satisfies $(iii)$. 
By monotonicity, the entries on the upper right part of $P$ are larger or equal to $1/2$ because it is the value on the diagonal. Therefore, $i<j$ implies $p_{ij}\leq 1/2$, i.e., $T_i\leq_P T_j$. \qedhere

\end{proof}

\begin{corollary}\label{coro:new_joe}
Assume that $\leq_P$ satisfies the property of SST for $P\in \mathcal{G}'_n$. 
Then for all $i,  j\le n$,
\[(T_i\leq_P T_j) \Longleftrightarrow (T_i\leq_x T_j).\] 
\end{corollary}

\begin{proof}
With the SST both $\leq_P$ and $\leq_x$ are total preorder on the set of teams. 
Therefore, it suffices to prove for every $i,j$,
the relation $T_i<_P\ T_j$ implies $T_i<_x T_j$,
and $T_i=_P T_j$ implies $T_i=_x T_j$. 
In fact, $T_i=_P T_j$ means that $p_{ij}=1/2$. 
By the SST, it implies $p_{ik}=p_{jk}$ for every $k\leq n$. 
Hence, $T_i=_x T_j$. 
Now assume $T_i<_P T_j$. 
Up to a proper relabelling corresponding to $(iii)$ in Proposition \ref{pro:equiv_strong},
it implies $i<j$. 
Thus, the elements of the $i^{\text{th}}$ row are entrywise smaller or equal to those of the $j^{\text{th}}$ row. 
But they are not all equal because $p_{ij}<1/2=p_{jj}$,
Thus, we get $T_i<_x T_j$.\qedhere
\end{proof}

The following example shows that if $P = (p_{ij})_{1\le i\neq j\le n} \in \mathcal{G}_n$ is only monotonic, 
the relation $\leq_P$ may fail to satisfy the SST, and be equivalent to $\leq_x$.

\begin{example}\label{ex:nonstrong}
Take $p_{12}=0.3,\  p_{13}=0,\ p_{23}=0.6$, and symmetrically $p_{21}=0.7,\  p_{31}=1,\ p_{32}=0.4$. 
This matrix is monotonic in $\mathcal{G}_3$ (but not in $\mathcal{G}'_3$). 
The coefficients $p_{23},\ p_{31},\ p_{21}$ are larger than $1/2$, but $p_{23}=0.6\geq \max(p_{21},p_{13})$ is not satisfied. 
So $\leq_P$ does not satisfy the property of SST.
Furthermore, $\leq_P$ is transitive with $T_1\ \leq_P\ T_3\ \leq_P\ T_2$, which is different from $T_1\ \leq_x\ T_2\ \leq_x\ T_3$. Note also that there is no relabelling that makes $P$ monotonic in $\mathcal{G}'_3$.
\end{example}

The following example show that even if $P=(p_{ij})_{i\neq j}\in \mathcal{G}_n$ is both strongly transitive and monotonic,
the order of the indices can be different from $\leq_P$ and $\leq_x$.

\begin{example}
For $n\in 2\N_{>0}$, take $p_{2k,2k-1}=0.4$, $p_{2k-1,2k}=0.6$ and the other entries $0$ or $1$ with $p_{ij}=0$ if $i<j$. The order $\leq_P$ ranks the teams as follows:
\[(T_2\ \leq_P\ T_1)\ \leq_P\ (T_3\ \leq_P\ T_4)\ \leq_P\cdots \leq_P\ (T_{2k}\leq_P\ T_{2k-1})\leq_P\cdots \leq_P\ (T_{n}\leq_P\ T_{n-1}).\]
 Still $(p_{ij})_{i\neq j}\in \mathcal{G}_n$ is monotonic.
\end{example}

{\bf Acknowledgments.}
We thank two anonymous referees for their careful reading and constructive suggestions which improved the presentation of this article.
We thank David Aldous, Brett Kolesnik and Erhan Bayraktar for helpful discussions.
Juillet thanks Tom Muller and Franck-kabrel Ngoupe for numerical experiments.

\medskip
{\bf Funding.}
This research was funded, in whole or in part, by the Agence Nationale de la Recherche (ANR), Grant ANR-23-CE40-0017. A CC-BY public copyright license has been applied by the authors to the present document and will be applied to all subsequent versions up to the Author Accepted Manuscript arising from this submission, in accordance with the grant’s open access conditions. Guo is supported by Cai Yuanpei Découverte 2024 and Bourse \emph{``Systemic Robustness and Systemic Failure''} provided by the Institut Europlace de Finance.  Tang acknowledges financial support by NSF grant DMS-2206038, the Columbia Innovation Hub grant,
and the Tang Family Assistant Professorship.

\bibliographystyle{abbrv}
\bibliography{unique}

\begin{thebibliography}{10}

\bibitem{acciaio2025calibration}
B.~Acciaio, A.~Marini, and G.~Pammer.
\newblock Calibration of the bass local volatility model.
\newblock {\em SIAM Journal on Financial Mathematics}, 16(3):803--833, 2025.

\bibitem{YF17}
I.~Adler, Y.~Cao, R.~Karp, E.~A. Pek{\"o}z, and S.~M. Ross.
\newblock Random knockout tournaments.
\newblock {\em Oper. Res.}, 65(6):1589--1596, 2017.

\bibitem{Aldous21}
D.~J. Aldous.
\newblock A prediction tournament paradox.
\newblock {\em Amer. Statist.}, 75(3):243--248, 2021.

\bibitem{AK22}
D.~J. Aldous and B.~Kolesnik.
\newblock To stay discovered: on tournament mean score sequences and the
  {B}radley-{T}erry model.
\newblock {\em Stochastic Process. Appl.}, 150:844--852, 2022.

\bibitem{Av98}
M.~Avellaneda.
\newblock Minimum entropy calibration of asset pricing models, internat.
\newblock {\em J. Theoret. Appl. Finance}, 1:447472, 1998.

\bibitem{AH81}
R.~Axelrod and W.~D. Hamilton.
\newblock The evolution of cooperation.
\newblock {\em Science}, 211(4489):1390--1396, 1981.

\bibitem{backhoff2017martingale}
J.~Backhoff, M.~Beiglb{\"o}ck, M.~Huesmann, and S.~K{\"a}llblad.
\newblock Martingale benamou--brenier: a probabilistic perspective.
\newblock {\em Ann. Probab.}, 48(5):2258--2289, 2020.

\bibitem{BDK24}
M.~Bassan, S.~Donderwinkel, and B.~Kolesnik.
\newblock Tournament score sequences, {E}rd\"{o}s-{G}inzburg-{Z}iv numbers, and
  the {L}\'evy-{K}hintchine method.
\newblock 2024.
\newblock arXiv:2407.01441.

\bibitem{BaLe00}
H.~H. Bauschke and A.~S. Lewis.
\newblock Dykstra's algorithm with {B}regman projections: a convergence proof.
\newblock {\em Optimization}, 48(4):409--427, 2000.

\bibitem{BDN23}
E.~Bayraktar, S.~Deng, and D.~Norgilas.
\newblock A potential-based construction of the increasing supermartingale
  coupling.
\newblock {\em Ann. Appl. Probab.}, 33(5):3803--3834, 2023.

\bibitem{BeHuSt16}
M.~Beiglb\"ock, M.~Huesmann, and F.~Stebegg.
\newblock Root to {K}ellerer.
\newblock In {\em S\'eminaire de {P}robabilit\'es {XLVIII}}, volume 2168 of
  {\em Lecture Notes in Math.}, pages 1--12. Springer, Cham, 2016.

\bibitem{BeJu16}
M.~Beiglb\"ock and N.~Juillet.
\newblock On a problem of optimal transport under marginal martingale
  constraints.
\newblock {\em Ann. Probab.}, 44(1):42--106, 2016.

\bibitem{BeJu21}
M.~Beiglb\"ock and N.~Juillet.
\newblock Shadow couplings.
\newblock {\em Trans. Amer. Math. Soc.}, 374(7):4973--5002, 2021.

\bibitem{benamou2015iterative}
J.-D. Benamou, G.~Carlier, M.~Cuturi, L.~Nenna, and G.~Peyr\'e.
\newblock Iterative {B}regman projections for regularized transportation
  problems.
\newblock {\em SIAM J. Sci. Comput.}, 37(2):A1111--A1138, 2015.

\bibitem{BV04}
S.~P. Boyd and L.~Vandenberghe.
\newblock {\em Convex optimization}.
\newblock Cambridge University Press, 2004.

\bibitem{Brad76}
R.~A. Bradley.
\newblock Science, statistics, and paired comparisons.
\newblock {\em Biometrics}, 32(2):213--232, 1976.

\bibitem{BT52}
R.~A. Bradley and M.~E. Terry.
\newblock Rank analysis of incomplete block designs: I. the method of paired
  comparisons.
\newblock {\em Biometrika}, 39(3/4):324--345, 1952.

\bibitem{Bru84}
R.~A. Brualdi.
\newblock The doubly stochastic matrices of a vector majorization.
\newblock {\em Linear Algebra Appl.}, 61:141--154, 1984.

\bibitem{BrHwPy97}
R.~A. Brualdi, S.-G. Hwang, and S.-S. Pyo.
\newblock Vector majorization via positive definite matrices.
\newblock {\em Linear Algebra Appl.}, 257:105--120, 1997.

\bibitem{CFM}
P.~Cartier, J.~M.~G. Fell, and P.-A. Meyer.
\newblock Comparaison des mesures port\'ees par un ensemble convexe compact.
\newblock {\em Bull. Soc. Math. France}, 92:435--445, 1964.

\bibitem{ChWo92}
K.-M. Chao and C.~S. Wong.
\newblock Applications of {{\(M\)}}-matrices to majorization.
\newblock {\em Linear Algebra Appl.}, 169:31--40, 1992.

\bibitem{Chen24}
H.~Chen, H.~Zhao, H.~Lam, D.~Yao, and W.~Tang.
\newblock Mallows{P}{O}: Fine-tune your {L}{L}{M} with preference dispersions.
\newblock In {\em ICLR}, volume~13, 2025.

\bibitem{CDF23}
A.~Claesson, M.~Dukes, A.~F. Frankl\'in, and S.~u.~O. Stef\'ansson.
\newblock Counting tournament score sequences.
\newblock {\em Proc. Amer. Math. Soc.}, 151(9):3691--3704, 2023.

\bibitem{conze2021bass}
A.~Conze and P.~Henry-Labordere.
\newblock Bass construction with multi-marginals: Lightspeed computation in a
  new local volatility model.
\newblock {\em Available at SSRN 3853085}, 2021.

\bibitem{Dv59}
H.~A. David.
\newblock Tournaments and paired comparisons.
\newblock {\em Biometrika}, 46(1-2):139--149, 1959.

\bibitem{MH99}
H.~De~March and P.~Henry-Labordere.
\newblock Building arbitrage-free implied volatility: Sinkhorn's algorithm and
  variants.
\newblock {\em arXiv preprint arXiv:1902.04456}, 2019.

\bibitem{DK24}
S.~Donderwinkel and B.~Kolesnik.
\newblock Tournaments and random walks.
\newblock 2024.
\newblock arXiv:2403.12940.

\bibitem{Ford57}
L.~R. Ford, Jr.
\newblock Solution of a ranking problem from binary comparisons.
\newblock {\em Amer. Math. Monthly}, 64(8):28--33, 1957.

\bibitem{GS71}
R.~L. Graham and J.~H. Spencer.
\newblock A constructive solution to a tournament problem.
\newblock {\em Canad. Math. Bull.}, 14:45--48, 1971.

\bibitem{HM66}
F.~Harary and L.~Moser.
\newblock The theory of round robin tournaments.
\newblock {\em Amer. Math. Monthly}, 73:231--246, 1966.

\bibitem{HLP}
G.~H. Hardy, J.~E. Littlewood, and G.~P\'olya.
\newblock {\em Inequalities}.
\newblock Cambridge, at the University Press,, 1952.
\newblock 2d ed.

\bibitem{HP11}
F.~Hirsch, C.~Profeta, B.~Roynette, and M.~Yor.
\newblock {\em Peacocks and associated martingales, with explicit
  constructions}, volume~3 of {\em Bocconi \& Springer Series}.
\newblock Springer, Milan; Bocconi University Press, Milan, 2011.

\bibitem{IsIyMcK00}
M.~Isaev, T.~Iyer, and B.~D. McKay.
\newblock Asymptotic enumeration of orientations of a graph as a function of
  the out-degree sequence.
\newblock {\em Electron. J. Comb.}, 27(1):research paper p1.26, 30, 2020.

\bibitem{Joe88}
H.~Joe.
\newblock Majorization, entropy and paired comparisons.
\newblock {\em Ann. Statist.}, 16(2):915--925, 1988.

\bibitem{Ju_seminaire}
N.~Juillet.
\newblock Peacocks parametrised by a partially ordered set.
\newblock In {\em S\'eminaire de {P}robabilit\'es {XLVIII}}, volume 2168 of
  {\em Lecture Notes in Math.}, pages 13--32. Springer, Cham, 2016.

\bibitem{Ju16}
N.~Juillet.
\newblock Stability of the shadow projection and the left-curtain coupling.
\newblock {\em Ann. Inst. Henri Poincar\'e{} Probab. Stat.}, 52(4):1823--1843,
  2016.

\bibitem{Ke73}
H.~G. Kellerer.
\newblock Integraldarstellung von {D}ilationen.
\newblock In {\em Transactions of the {S}ixth {P}rague {C}onference on
  {I}nformation {T}heory, {S}tatistical {D}ecision {F}unctions, {R}andom
  {P}rocesses ({T}ech. {U}niv. {P}rague, {P}rague, 1971; dedicated to the
  memory of {A}nton\'in \v Spa\v cek)}, pages 341--374. Academia [Publishing
  House of the Czechoslovak Academy of Sciences], Prague, 1973.

\bibitem{Kendall55}
M.~G. Kendall.
\newblock Further contributions to the theory of paired comparisons.
\newblock {\em Biometrics}, 11(1):43--62, 1955.

\bibitem{kolesnik2023asymptotic}
B.~Kolesnik.
\newblock The asymptotic number of score sequences.
\newblock {\em Combinatorica}, 43(4):827--844, 2023.

\bibitem{kolesnik2023coxeter}
B.~Kolesnik and M.~Sanchez.
\newblock Coxeter tournaments.
\newblock {\em arXiv preprint arXiv:2302.14002}, 2023.

\bibitem{Landau53}
H.~G. Landau.
\newblock On dominance relations and the structure of animal societies. {III}.
  {T}he condition for a score structure.
\newblock {\em Bull. Math. Biophys.}, 15:143--148, 1953.

\bibitem{La97}
J.-F. Laslier.
\newblock {\em Tournament solutions and majority voting}, volume~7.
\newblock Springer, 1997.

\bibitem{LR81}
E.~P. Lazear and S.~Rosen.
\newblock Rank-order tournaments as optimum labor contracts.
\newblock {\em J. Political Econ.}, 89(5):841--864, 1981.

\bibitem{Low}
G.~Lowther.
\newblock Limits of one-dimensional diffusions.
\newblock {\em Ann. Probab.}, 37(1):78--106, 2009.

\bibitem{Luce}
R.~D. Luce.
\newblock {\em Individual Choice Behavior a Theoretical Analysis}.
\newblock John Wiley \& Sons Inc., 1959.

\bibitem{Mac20}
P.~A. MacMahon.
\newblock An american tournament treated by the calculus of symmetric
  functions.
\newblock {\em Quart. J. Math}, 49:1--36, 1920.

\bibitem{Mallows57}
C.~L. Mallows.
\newblock Non-null ranking models. {I}.
\newblock {\em Biometrika}, 44(1/2):114--130, 1957.

\bibitem{MOA}
A.~W. Marshall, I.~Olkin, and B.~C. Arnold.
\newblock {\em Inequalities: theory of majorization and its applications}.
\newblock Springer Series in Statistics. Springer, New York, second edition,
  2011.

\bibitem{MN83}
P.~McCullagh and J.~A. Nelder.
\newblock {\em Generalized linear models}.
\newblock Monographs on Statistics and Applied Probability. Chapman \& Hall,
  London, 1983.

\bibitem{Moon63}
J.~W. Moon.
\newblock An extension of {L}andau's theorem on tournaments.
\newblock {\em Pacific J. Math.}, 13:1343--1345, 1963.

\bibitem{Moon68}
J.~W. Moon.
\newblock {\em Topics on tournaments}.
\newblock Holt, Rinehart and Winston, New York-Montreal, Que.-London, 1968.

\bibitem{MoPu70}
J.~W. Moon and N.~J. Pullman.
\newblock On generalized tournament matrices.
\newblock {\em SIAM Rev.}, 12:384--399, 1970.

\bibitem{Mos51}
F.~Mosteller.
\newblock Remarks on the method of paired comparisons: I{I}{I}. a test of
  significance for paired comparisons when equal standard deviations and equal
  correlations are assumed.
\newblock {\em Psychometrika}, 16(2):207--218, 1951.

\bibitem{MR01}
A.~M\"uller and L.~R\"uschendorf.
\newblock On the optimal stopping values induced by general dependence
  structures.
\newblock {\em J. Appl. Probab.}, 38(3):672--684, 2001.

\bibitem{MuSt02}
A.~M\"uller and D.~Stoyan.
\newblock {\em Comparison methods for stochastic models and risks}.
\newblock Wiley Series in Probability and Statistics. John Wiley \& Sons, Ltd.,
  Chichester, 2002.

\bibitem{PC19}
G.~Peyr{\'e} and M.~Cuturi.
\newblock Computational optimal transport: With applications to data science.
\newblock {\em Foundations and Trends{\textregistered} in Machine Learning},
  11(5-6):355--607, 2019.

\bibitem{PP24}
P.~Pinto and N.~Pischke.
\newblock On {D}ykstra’s algorithm with {B}regman projections.
\newblock 2024.
\newblock Available at
  \url{https://www2.mathematik.tu-darmstadt.de/~pinto/articles/B-Dykstra.pdf}.

\bibitem{Plackett}
R.~Plackett.
\newblock The analysis of permutations.
\newblock {\em Appl. Stat.}, pages 193--202, 1975.

\bibitem{DPO}
R.~Rafailov, A.~Sharma, E.~Mitchell, S.~Ermon, C.~D. Manning, and C.~Finn.
\newblock Direct preference optimization: Your language model is secretly a
  reward model.
\newblock In {\em Neurips}, volume~36, 2023.

\bibitem{SS07}
M.~Shaked and J.~G. Shanthikumar.
\newblock {\em Stochastic orders}.
\newblock Springer Series in Statistics. Springer, New York, 2007.

\bibitem{SK67}
R.~Sinkhorn and P.~Knopp.
\newblock Concerning nonnegative matrices and doubly stochastic matrices.
\newblock {\em Pacific J. Math.}, 21:343--348, 1967.

\bibitem{ST59}
H.~Steinhaus and S.~Trybu\l{}a.
\newblock On a paradox in applied probabilities.
\newblock {\em Bull. Acad. Polon. Sci}, 7(67-69):108, 1959.

\bibitem{Stob84}
M.~Stob.
\newblock A supplement to: ``{A} mathematician's guide to popular sports''\
  [{A}mer. {M}ath. {M}onthly {\bf 90}\ (1983), no. 4, 246--266; {MR}0700265
  (85i:05122a)]\ by {T}. {J}ech.
\newblock {\em Amer. Math. Monthly}, 91(5):277--282, 1984.

\bibitem{Sto23}
P.~K. Stockmeyer.
\newblock Counting various classes of tournament score sequences.
\newblock {\em J. Integer Seq.}, 26(5):Art. 23.5.2, 9, 2023.

\bibitem{Str65}
V.~Strassen.
\newblock The existence of probability measures with given marginals.
\newblock {\em Ann. Math. Statist.}, 36:423--439, 1965.

\bibitem{Su02}
R.~Suck.
\newblock Independent random utility representations.
\newblock {\em Math. Soc. Sci.}, 43(3):371--389, 2002.

\bibitem{Laj91}
L.~Tak\'acs.
\newblock A {B}ernoulli excursion and its various applications.
\newblock {\em Adv. in Appl. Probab.}, 23(3):557--585, 1991.

\bibitem{Thur27}
L.~Thurstone.
\newblock A law of comparative judgment.
\newblock {\em Psychological review}, 34(4):273, 1927.

\bibitem{Thur31}
L.~Thurstone.
\newblock Rank order as a psycho-physical method.
\newblock {\em Journal of Experimental Psychology}, 14(3):187, 1931.

\bibitem{Tr60}
S.~Trybu\l{}a.
\newblock On the paradox of three random variables.
\newblock {\em Zastos. Mat.}, 5:321--332, 1960/61.

\bibitem{VuHi21}
P.~Vuksanovic and A.~J. Hildebrand.
\newblock On cyclic and nontransitive probabilities.
\newblock {\em Involve}, 14(2):327--348, 2021.

\bibitem{Zer29}
E.~Zermelo.
\newblock Die {B}erechnung der {T}urnier-{E}rgebnisse als ein {M}aximumproblem
  der {W}ahrscheinlichkeitsrechnung.
\newblock {\em Math. Z.}, 29(1):436--460, 1929.

\end{thebibliography}
\end{document}